\documentclass[reqno,10pt,centertags,draft]{amsart}
\usepackage{amsmath,amsthm,amscd,amssymb,latexsym,upref}
\date{\today}

\newcommand{\bbD}{{\mathbb{D}}}

\newcommand{\bbR}{{\mathbb{R}}}

\newcommand{\bbZ}{{\mathbb{Z}}}
\newcommand{\bbC}{{\mathbb{C}}}

\newcommand{\cH}{{\mathcal{H}}}
\newcommand{\bbT}{{\mathbb{T}}}
\newcommand{\mcB}{{\mathcal{B}}}
\newcommand{\cZ}{{\mathcal{Z}}}
\newcommand{\cF}{{\mathcal{F}}}
\newcommand{\cP}{{\mathcal{P}}}
\newcommand{\cQ}{{\mathcal{Q}}}
\newcommand{\cR}{{\mathcal{R}}}
\newcommand{\cS}{{\mathcal{S}}}
\newcommand{\cE}{{\mathcal{E}}}
\newcommand{\cG}{{\mathcal{G}}}
\newcommand{\cM}{{\mathcal{M}}}
\newcommand{\bW}{{\mathbf{W}}}
\newcommand{\bB}{{\mathbf{B}}}
\newcommand{\bH}{{\mathbf{H}}}
\newcommand{\bT}{{\mathbf{T}}}

\newcommand{\fe}{{\mathfrak{e}}}
\newcommand{\fw}{{\mathfrak{w}}}
\newcommand{\fu}{{\mathfrak{u}}}

\newcommand{\fA}{{\mathfrak{A}}}
\newcommand{\fH}{{\mathfrak{H}}}
\renewcommand{\t}{\tau}
\renewcommand{\k}{\varkappa}
\newcommand{\z}{\zeta}
\renewcommand{\Re}{\text{\rm Re}}
\renewcommand{\Im}{\text{\rm Im}}

\newcommand{\tr}{\text{\rm tr}}
\newcommand{\fae}{{\mathfrak{A}}_{SB}(E)}
\newcommand{\thte}{{\Theta}_{SB}(E)}
\newcommand{\ale}{{A}_{SB}(E)}

\newcommand{\Mth}{\langle \theta \rangle_{I}}

\newcommand{\Mst}{\langle \tilde R_+ \rangle_{I}}
\allowdisplaybreaks \numberwithin{equation}{section}
\newtheorem{theorem}{Theorem}[section]

\newtheorem{lemma}[theorem]{Lemma}
\newtheorem{proposition}[theorem]{Proposition}
\newtheorem{corollary}[theorem]{Corollary}

\theoremstyle{definition}
\newtheorem{definition}[theorem]{Definition}
\newtheorem{remark}[theorem]{Remark}

\newtheorem{example}[theorem]{Example}

\begin{document}

\title[]{CMV matrices with asymptotically constant coefficients. Szeg\"o over Blaschke class, Scattering Theory}
\author[]{F. Peherstorfer, A. Volberg,
and P. Yuditskii}

\address{Institute for Analysis, Johannes Kepler University Linz,
A-4040 Linz, Austria}
\email{Franz.Peherstorfer@jku.at}
\email{Petro.Yudytskiy@jku.at}
\thanks{
Partially supported by NSF grant DMS-0501067  the Austrian Founds
FWF, project number: P20413--N18}

\address{
Department of Mathematics, Michigan State University, East Lansing,
MI 48824, USA} \email{volberg@math.msu.edu}

\date{\today}

\begin{abstract}
We develop a modern extended scattering theory for CMV matrices with
asymptotically  constant Verblunsky  coefficients. First we
represent CMV matrices with constant coefficients as a
multiplication operator in $L^2$ -space with respect to a specific
basis. This basis substitutes the standard basis in $L^2$, which is
used for the free Jacobi matrix. Then we demonstrate that a similar
orthonormal system in a certain "weighted'' Hilbert space, which we
call the Fadeev-Marchenko (FM) space, behaves asymptotically as the
system in the standard (free) case discussed just before. The
duality between the two types of Hardy subspaces in it plays the key
role in the proof of all asymptotics involved.  We show that the
traditional (Faddeev-Marchenko) condition is too restrictive to
define the class of CMV matrices for which there exists a unique
scattering representation. The main results are: 1) Szeg\"o-Blaschke
class: the class of twosided CMV matrices acting in $l^2$ , whose
spectral density satisfies the Szeg\"o condition and whose point
spectrum the Blaschke condition, corresponds precisely to the class
where the scattering problem can be posed and solved. That is, to a
given CMV matrix of this class, one can associate the scattering
data and related to them the FM space. The CMV matrix corresponds to
the multiplication operator in this space, and the orthonormal basis
in it (corresponding to the standard basis in $l^2$) behaves
asymptotically as the basis associated with the free system. 2)
$A_2$-Carleson class: from the point of view of the scattering
problem, the most natural class of CMV matrices is that one in which
a) the scattering data determine the matrix uniquely and b) the
associated Gelfand-Levitan Marchenko transformation operators are
bounded. Necessary and sufficient conditions for this class can be
given in terms of an $A_2$ kind condition for the density of the
absolutely continuous spectrum and a Carleson kind condition for the
discrete spectrum. Similar close to the optimal conditions are given
directly in terms of the scattering data.
\end{abstract}

\maketitle

\section{Introduction}

\subsection{Why almost periodic operators are so important?} One of
the possible reason is the following rough statement: all operators with
the "strong" absolutely continuous spectrum asymptotically looks as
almost periodic operators with the same a.c. spectral set. To formulate a
theorem (in support of the above "philosophy") let us give some
definitions.

Of course
$$
a_n=\cos(\kappa n+\alpha)
$$
is the best known example of an almost periodic sequence. Generally,
they are of the form
$$
a_n=f(\kappa n),
$$
where $f$ is a continuous function on a compact Abelian group $\cG$
and $\kappa\in \cG$.

A Jacobi matrix
$$
J=J(\{p_n\},\{q_n\})=\begin{bmatrix}
\ddots&\ddots&\ddots& &\\
&p_0&q_0&p_1&\\
& &p_1&q_1&p_2& \\
& & & \ddots& \ddots& \ddots\\
\end{bmatrix}
$$
with almost periodic coefficient sequences $\{p_n\}_{n\in\bbZ}$ and
$\{q_n\}_{n\in\bbZ}$ is called almost periodic. In what follows we
assume $p_n>0$ and $q_n\in\bbR$.

Let $E\subset \bbR$ be  a compact of positive Lebesgue measure.
Moreover, we require a certain kind of {\it regularity} of the set
with respect to the Lebesgue measure: there exists $\eta>0$ such
that
\begin{equation}\label{hom}
    |E\cap(x-\delta,x+\delta)|\ge \eta\delta\quad\text{for all}\
    x\in E \ \text{and}\ 0<\delta<1.
\end{equation}
Such set $E$ is called homogeneous \cite{car}.

Denote
\begin{equation}\label{ape}
    J(E)=\{J \ \text{is almost
    periodic}:\sigma(J)=\sigma_{a.c.}(J)=E\}.
\end{equation}
That is $J(E)$ is the collection of all almost periodic Jacobi
matrices acting in $l^2(\bbZ)$, such that the spectral set of $J$ is
$E$, moreover, the support of the absolutely continuous component of
the spectrum covers the whole $E$.  It is important that the above
set has a parametric representation
\begin{theorem}\cite{syu} Let $\pi_1(\Omega)$ be the fundamental
group of the domain $\Omega:=\bar\bbC\setminus E$ and
$\pi^*_1(\Omega)$ denote its  group of characters. There exist the
character $\mu\in\pi^*_1(\Omega)$ and the continuous functions
$\cP(\alpha)$ and $\cQ(\alpha)$, $\alpha\in\pi^*_1(\Omega)$, such
that
\begin{equation}\label{repralpha}
    J(E)=\{J(\alpha),\ \alpha\in\pi^*_1(\Omega):p_n(\alpha)=\cP(\alpha\mu^{-n}), \
    q_n(\alpha)=\cQ(\alpha\mu^{-n})\}.
\end{equation}
\end{theorem}

\begin{remark}  $\pi^*_1(\Omega)$ is a  compact (multiplicative) Abelian
group. In \cite{syu} precise formulae for $\cP(\alpha), \cQ(\alpha)$
and $\mu$ are given.
\end{remark}

Recall that one--sided Jacobi matrices $J_+$ acting as bounded
self--adjoint operators in $l^2(\bbZ_+)$ are in one--to--one
correspondence with compactly supported measures on $\bbR$: $J_+$ is
the matrix of the multiplication operator by independent variable
with respect to the basis of  orthonormal  polynomials
$\{P_n(z;\sigma)\}_{n=0}^\infty $ in $L^2_{\sigma}$, the measure $\sigma$ is
called the spectral measure of $J_+$.

Let $\omega(dx)=\omega(dx,\infty;\Omega)$ and
$G(z)=G(z,\infty;\Omega)$ be the harmonic measure on $E$ and the
Green function in the domain $\Omega$ with respect to infinity.
\begin{theorem}\label{pyuth}\cite{pyu} Let $E$ be a homogenous set and $X$ be a
discreet subset   of $\bbR\setminus E$, consisting of points
accumulating to $E$ only, moreover
\begin{equation}\label{blashke}
    \sum_{x_j\in X} G(x_j)<\infty.
\end{equation}
Assume that $\sigma$ is a measure on $E\cup X$ such that
\begin{equation}\label{szego}
    \int_{E}|\log\sigma'_{a.c.}(x)|\omega(dx)<\infty,
\end{equation}
and $J_+$ is the Jacobi matrix associated to the $\sigma$.

Then there exists $\alpha\in\pi_1^*(\Omega)$ such that
\begin{equation}\label{claimpyu}
    p_n-\cP(\alpha\mu^{-n})\to 0,\quad  q_n-\cQ(\alpha\mu^{-n})\to
0.
\end{equation}
\end{theorem}

\begin{remark} In \cite{pyu} asymptotic for $P_n(z;\sigma)$ is also
given.
\end{remark}

In this paper we present a similar to Theorem \ref{pyuth} result for
CMV matrices in the basic case, when the resolvent set is a
simply--connected domain. The general case of a multi--connected
domain with a homogeneous boundary will be considered in a
forthcoming paper (we would like to mention this here, in
particular, to indicate the power of our method).

\subsection{CMV matrices with constant coefficients}
For
 a given sequence of numbers from the unit disk $\bbD$
\begin{equation}\label{21o1}
...,\,a_{-1},\,a_{0},\,a_{1},\,a_{2},\,...
\end{equation}
define unitary $2\times 2$ matrices
\begin{equation*}
A_j=\begin{bmatrix} \overline{a_j}& \rho_i\\
\rho_j&-a_j
\end{bmatrix}, \quad \rho_j=\sqrt{1-|a_j|^2},
\end{equation*}
and unitary operators in $l^2(\bbZ)=l^2(\bbZ_-)\oplus l^2(\bbZ_+)$
given by block--diagonal matrices
\begin{equation*}
\mathfrak A_0=\begin{bmatrix} \ddots& & &\\
& A_{-2}& & \\
& & A_{0}&  \\
 & & &\ddots
\end{bmatrix}, \quad
\mathfrak A_1=\cS\begin{bmatrix} \ddots& & &\\
& A_{-1}& & \\
& & A_{1}&  \\
 & & &\ddots
\end{bmatrix}\cS^{-1},
\end{equation*}
where $\cS|j\rangle=|j+1\rangle$.
 The CMV matrix $\mathfrak A$,
generated by the sequence \eqref{21o1}, is the product
\begin{equation}\label{21o2}
\mathfrak A=\mathfrak A(\{a_j\}):=\mathfrak A_0 \mathfrak A_1,
\end{equation}
 CMV matrices have been studied by
M.J. Cantero, L. Moral, and L. Vel\'asquez \cite{cmv}, for
historical details see \cite{s5y}.

Recall that a Schur function $\theta_+(v)$, $|\theta_+(v)|\le 1$,
$v\in\bbD$, (a finite Blaschke product is a special case) is in a
one to one correspondence with the so called Schur parameters
\begin{equation}\label{21o3}
\theta_+(v)\sim\{a_0, a_1,...\},
\end{equation}
where
\begin{equation*}
\theta_+(v)=\frac{a_0+v \theta_+^{(1)}(v)}{1+v\overline{a_0}
\theta_+^{(1)}(v)},
\end{equation*}
$a_0=\theta_+(0)$, $a_1= \theta_+^{(1)}(0)$, and so on...

It is evident that the matrix $\mathfrak A(\{a_k\})$ is well defined
by the two
 Schur functions $\{\theta_+(v), \theta_-(v)\}$, given by  \eqref{21o3} and
\begin{equation}\label{27o4}
\theta_-(v)\sim\{-\overline{a_{-1}},-\overline{a_{-2}},...\}.
\end{equation}

The spectral set of a CMV matrix $\mathfrak A_a$ with the constant
coefficients $a_n=a\not=0$ is an arc
\begin{equation}\label{spcnst}
    E=E_a=\{e^{i\xi}:\xi_0\le\xi\le 2\pi-\xi_0\},
\end{equation}
where $\rho=\sqrt{1-|a|^2}=\cos\frac{\xi_0} 2$.

The following construction is a very special case of a functional
realization of almost periodic operators \cite{mvm, pyu2}. The
domain $\Omega=\bar\bbC\setminus E$ is conformally equivalent to the
unit disk $\bbD$:
\begin{equation}\label{CsmE}
\begin{split}
    z=&i\frac{1+v}{1-v},\ v\in\Omega,\\
2\tan \frac{\xi_0}2\, z=&\z+\frac 1{\z},\ \z\in\bbD.
\end{split}
\end{equation}
Put $\k:=-i\tan\frac{\pi-\xi_0}4\in\bbD$, so that
\begin{equation}\label{kappa}
    v(\k)=0,\quad v(\bar\k)=\infty.
\end{equation}

The Green function $G(v,v_0)=G(v,v_0;\Omega)$ is of the form
\begin{equation}\label{greenom}
    G(v(\z),v(\z_0))=\log\frac 1{|b_{\z_0}(\z)|},
\end{equation}
where
\begin{equation}\label{blfac}
b_{\z_0}(\z)=e^{ic}\frac{\z-\z_0}{1-\z\bar\z_0}
\end{equation}
is the Blaschke factor in $\bbD$. For $\Im \z_0\not=0$ it is
convenient to use the normalization $b_{\z_0}(\bar\z_0)>0$. In
particular,
\begin{equation}\label{blfackappa}
b_{\k}(\z)=\frac 1 i\frac{\z-\k}{1-\z\bar\k} \ \text{and threfore}\
v(\z)=\frac{b_{\k}(\z)}{b_{\bar\k}(\z)}.
\end{equation}

For the Lebesgue measure $dm=dm(\t)$ on the unit circle $\bbT$ as
usual we define the $L^2$--norm by
\begin{equation}\label{1.9d}
||f||^2=\int_{\bbT}|f(\t)|^2\,dm(\t),
\end{equation}
so that the reproducing kernel of the $H^2$ subspace is of the form
$k_{\z_0}(\z)=k(\zeta,\zeta_0)= \frac{1}{1-\z\bar\z_0}$. By
$K(\zeta,\zeta_0)$ we denote the normalized kernel
\begin{equation}\label{noke}
    K(\zeta,\zeta_0)=\frac{k_{\z_0}(\z)}{\Vert
k_{\z_0}\Vert}=\frac{\sqrt{1-|\z_0|^2}}{1-\z\bar\z_0}.
\end{equation}

Using the above notation, to the given $\beta_k=e^{2\pi i\tau_k}\in
\bbT$, $k=0,1$, we associate the space $H^2(\beta_0,\beta_1)$ of
analytic multivalued functions $f(\zeta)$,
$\zeta\in\bbD\setminus\{\k,\bar\k\}$, such that $|f(\zeta)|^2$ is
singlevalued and has a harmonic majorant and
$$
f\circ\gamma_i=\beta_i f,
$$
where $\gamma_i$ is a small circle around $\k$ and $\bar\k$
respactively. Such a space can be reduced to the standard Hardy
space $H^2$, moreover
$$
H^2(\beta_0,\beta_1)= b_{\k}^{\tau_0}b_{\bar\k}^{\tau_1}H^2.
$$

\begin{lemma} Let $b=\sqrt{b_{\k}b_{\bar\k}}$.  The space $bH^2(1,-1)$ is a subspace
of $H^2(-1,1)$ having a one  dimensional orthogonal compliment,
moreover
\begin{equation}\label{vs3}
H^2(-1,1)=\{\sqrt{b_{\k}}k_{\bar\k}\} \oplus bH^2(1,-1).
\end{equation}
\end{lemma}

Iterating, now, the decomposition \eqref{vs3}
\begin{equation*}
\begin{split}
H^2(-1,1)=&\{\sqrt{b_{\k}}k_{\bar\k}\} \oplus bH^2(1,-1)\\=&
\{\sqrt{b_{\k}}k_{\bar\k}\} \oplus b\{\sqrt{b_{\bar\k}}k_{\k}\}
\oplus b^2H^2(-1,1)=...,
\end{split}
\end{equation*}
one gets an orthogonal basis in $H^2(-1,1)$ consisting of vectors of
two sorts
\begin{equation}\label{stbasis}
    b^{2m}\{\sqrt{b_{\k}}k_{\bar\k}\}\ \ \text{and}\ \
b^{2m+1}\{\sqrt{b_{\bar\k}}k_{\k}\}.
\end{equation}
Note that this orthogonal system can be extended to the negative
integers $m$ so that we obtain a basis in the standard $L^2$.

\begin{theorem}\label{th1.6}
With respect to the orthonormal basis
\begin{equation}\label{ts4}
e_n=\begin{cases} b^{2m}\sqrt{b_{\k}}K_{\bar\k}e^{ic}
, &n=2m\\
b^{2m+1}\sqrt{b_{\bar\k}}K_{\k} ,&n=2m+1
\end{cases}
\end{equation}
the multiplication operator by $v$ is the CMV matrix $\fA_a$,
$a=e^{ic}\frac{k(\k,\bar\k)}{k(\k,\k)}=e^{ic}\sin\frac{\xi_0}2$.
\end{theorem}

As it was mentioned the above theorem is a part of a very general
construction. In this particular case, we will break the symmetry of
the shift operation. Paying this price we can use the standard $H^2$
space. Factoring out $\sqrt{b_{\k}}$, we get
\begin{theorem} The system of functins
\begin{equation}\label{ts4bis}
\fe_{n,c}=\begin{cases} b_\k^{m} b_{\bar\k}^{m}K_{\bar\k}e^{ic} ,
&n=2m\\
b_\k^{m} b_{\bar\k}^{m+1}K_{\k} ,&n=2m+1
\end{cases}
\end{equation}
forms orthonormal basis in $H^2$ if $n\in \bbZ_+$ and in $L^2$ if
$n\in \bbZ$. With respect to this basis the multiplication operator
by $v$ is the CMV matrix $\fA_a$.
\end{theorem}

\subsection{Szeg\"o over Blaschke calss and the direct scattering}
It follows from
\begin{equation}\label{3term}
\begin{split}
    \fA\{|2n-1\rangle\rho_{2n-1}-|2n\rangle\bar a_{2n-1}\}=&|2n\rangle\bar
a_{2n}+|2n+1\rangle\rho_{2n}\\
\fA^{-1}\{|2n\rangle\rho_{2n}-|2n+1\rangle a_{2n}\}=&|2n+1\rangle
a_{2n+1}+|2n+2\rangle\rho_{2n+1}
\end{split}
\end{equation}
that the subspace formed by the vectors $|-1\rangle, |0\rangle$ is
cyclic for $\fA$. The resolvent ma\-trix--function is defined by the
relation
\begin{equation}\label{rf}
    \cR(v)=\cE^*\frac{\fA+v}{\fA-v}\cE,
\end{equation}
where $\cE:\bbC^2\to l^2(\bbZ)$ in such a way that
\begin{equation*}
    \cE\begin{bmatrix}c_{-1}\\c_0\end{bmatrix}=|-1\rangle c_{-1}+|0\rangle c_{0}.
\end{equation*}
This matrix--function possesses the integral representation
\begin{equation}\label{ir}
    \cR(v)=\int_{\bbT}\frac{t+v}{t-v}\, d\Sigma(t)
\end{equation}
with $2\times 2$ matrix--measure. $\fA$ is unitary equivalent  to
the multiplication operator by an independent variable on
\begin{equation*}
    L^2_{d\Sigma}=\left\{f=\begin{bmatrix}f_{-1}(t)\\f_0(t)
\end{bmatrix}: \int_{\bbT}f^*(t)d\Sigma(t) f(t)<\infty\right\}.
\end{equation*}

Note that $\Sigma$ is not an arbitrary $2\times 2$ matrix--measure,
but has a specific structure. In terms of the Schur functions
\eqref{21o3}, \eqref{27o4}
\begin{equation}\label{Rwiththeta}
    \cR(v)=\frac{I+v A^*_{-1}
\begin{bmatrix}
\theta_-^{(1)}(v)&0\\
0&\theta_+(v)\end{bmatrix}}{I-v A^*_{-1}
\begin{bmatrix}
\theta_-^{(1)}(v)&0\\
0&\theta_+(v)\end{bmatrix}}.
\end{equation}
In particular, this means that the rang of the measure $\Sigma(t_0)$
at an isolated spectral point $t_0$ is one,  and the spectral
density has the form

\begin{equation}\label{Wwiththeta}
\frac{d\Sigma(t)}{dm(t)}=
 \frac{I+\cR^*(t)}2
\begin{bmatrix}
{1-|\theta_-^{(1)}(t)|^2}&0\\
0&{1-|\theta_+(t)|^2}\end{bmatrix}\frac{I+\cR(t)} 2.
\end{equation}

\begin{definition} Let $E$ be an arc of the form \eqref{spcnst} and $X$ be a discreet set in
$\bbT\setminus E$, which satisfies the Blaschke condition in
$\Omega=\bar\bbC\setminus E$:
\begin{equation}\label{matbl}
\begin{split}
   X=&\{t_k=v(\zeta_k): \z_k\in\cZ\},\\
 \cZ=&\{\z_k\in \bbR\cap\bbD: \sum(1-|\zeta_k|)<\infty\}.
\end{split}
\end{equation}
We say that $\fA$ is in $\fae$ if $\sigma(\fA)=E\cup X$, the
spectral measure is absolutely continuous  on $E$,
\begin{equation}\label{dencity}
    d\Sigma|E=W(t)\,dm(t)
\end{equation}
and the density satisfies the Szeg\"o condition
\begin{equation}\label{matrsz}
    \log\det W(v(\tau))\in L^1.
\end{equation}
\end{definition}

\begin{remark}Due to  condition \eqref{matbl}, $\cR(v)$ is of bounded
characteristic in $\Omega$ \cite[Theorem D]{syu}, see Sect
\ref{specshur}. Therefore \eqref{matrsz} is equivalent to
\begin{equation}\label{matszvar}
    \log(1-|\theta_-^{(1)}(v(\tau))|^2)(1-|\theta_+(v(\tau))|^2)\in
L^1.
\end{equation}

\end{remark}

With the set $X$ \eqref{matbl} we associate  the Blaschke product
\begin{equation}\label{9.9d}
B(\z)=\prod_{\cZ}\frac{|\z_k|} {\z_k}\frac{\z_k-\z} {1-\z\bar\z_k}
\end{equation}
(this product contains the factor $\z$ if $0\in \cZ$).

\begin{theorem}\label{thszego}
Let $\fA\in \fae$. Then there exists a generalized eigen vector (see
\eqref{3term})
\begin{equation}\label{evp}
\begin{split}
    v(\t)\{e^+(2m-1,\t)\rho_{2m-1}-&e^+(2m,\t)\bar a_{2m-1}\}\\=&e^+(2m,\t)
a_{2m}+e^+(2m+1,\t)\rho_{2m}\\
v(\t)^{-1}\{e^+(2m,\t)\rho_{2m}-&e^+(2m+1,\t)a_{2m}\}\\=&e^+(2m+1,\t)
a_{2m+1}+e^+(2m+2,\t)\rho_{2m+1}
\end{split}
\end{equation}
such that (see \eqref{ts4bis})
\begin{equation}\label{asppm}
\begin{split}
T_+(\t)e^+(-n-1,\t)&=\bar\t\fe_{n,c_-}(\bar
\t)+R_-(\t)\fe_{n,c_-}(\t)+o(1), \\
T_+(\t)e^+(n,\t)&=T_+(\t)\fe_{n,c_+}(\t)+o(1)
\end{split}
\end{equation}
in $L^2$ as $n\to\infty$. Moreover the functions
\begin{equation}\label{anevenodd}
   b^{-m}_{\k}b^{-m}_{\bar\k} (BT_+)(\t)e^+(2m,\t)\quad\text{and}\quad
 b^{-m}_{\k}b^{-m-1}_{\bar\k} (BT_+)(\t)e^+(2m+1,\t)
\end{equation}
belong to $H^2$ for all $m\in\bbZ$, that is,  $e^+(n,\z)$ is well
defined in $\bbD$ and
\begin{equation}\label{defnup}
    \frac
1{\nu_+(\z_k)}:=\sum_{n=-\infty}^{\infty}|e^+(n,\z_k)|^2<\infty
\end{equation}
for all $\z_k: v(\z_k)\in X$.
\end{theorem}

Let $\iota|n\rangle:=|-1-n\rangle$. The following involution acts on
the isospectral set of CMV  matrices
\begin{equation}\label{involution}
    \iota\fA(\{a_n\})\iota=\fA(\{-\overline{a_{-n-2}}\}).
\end{equation}
Theorem \ref{thszego} with respect to $\iota\fA\iota$ can be
rewritten into the form
\begin{corollary}\label{corszego}
Simultaneously with the eigen vector \eqref{evp} there exists the
vector
\begin{equation}\label{evm}
\begin{split}
    v(\t)\{e^-(-2m,\t)\rho_{2m-1}-&e^-(-2m-1,\t)\bar a_{2m-1}\}\\=&e^-(-2m-1,\t)
a_{2m}+e^-(-2m-2,\t)\rho_{2m}\\
v(\t)^{-1}\{e^-(-2m-1,\t)\rho_{2m}-&e^-(-2m-2,\t)a_{2m}\}\\=&e^-(-2m-2,\t)
a_{2m+1}+e^-(-2m-3,\t)\rho_{2m+1},
\end{split}
\end{equation}
possessing the asymptotics
\begin{equation}\label{asmpm}
\begin{split}
T_-(\t)e^-(-n-1,\t)&=\bar\t\fe_{n,c_+}(\bar
\t)+R_+(\t)\fe_{n,c_+}(\t)+o(1), \\
T_-(\t)e^-(n,\t)&=T_-(\t)\fe_{n,c_-}(\t)+o(1)
\end{split}
\end{equation}
in $L^2$ as $n\to\infty$. And, also,
\begin{equation}\label{defnum}
    \frac
1{\nu_-(\z_k)}:=\sum_{n=-\infty}^{\infty}|e^-(n,\z_k)|^2<\infty
\end{equation}
for all $\z_k: v(\z_k)\in X$.
\end{corollary}

\begin{remark}
1. We will see that Theorem \ref{thszego} belongs to the family of
Szeg\"o kind results. Actually it claims that an orthonormal system
in a certain "weighted" Hilbert space behaves asymptotically as such
a system in the "unweighted" space.

2. On the other hand it belongs to the family of "direct scattering
theorems" having the following specific: the class of CMV matrices
is given in terms of their spectral properties, but not in terms of
a behavior of the coefficients sequences.
\end{remark}

 $R_\pm$, $T_\pm$ in asymptotics \eqref{asppm},
\eqref{asmpm} are called the reflection and transmission
coefficients respectively. They form the, so called,  scattering
matrix
\begin{equation}\label{1.10d}
S(\t)=\begin{bmatrix}
R_-&T_-\\
T_+& R_+
\end{bmatrix}(\t),\ \t\in\bbT.
\end{equation}

\begin{proposition}\label{propscat}
The matrix function $S$ possesses two fundamental properties:
$S^*(\bar\tau)=S(\tau)$ and it is unitary--valued. The third
property is analyticity of the entries $T_\pm$, each of them has
analytic continuation in $\bbD$ as a function of bounded
characteristic of a specific nature,
--- it is the ratio of an {\it outer function} and a {\it Blaschke
product}. That is,
\begin{itemize}
\item
$R_+$ is a contractive symmetric Szeg\"o function on $\bbT$:
\begin{equation}\label{5.9d}
\begin{split}
&| R_+(\t)|\le 1,\quad R_+(\t)= \overline{R_+(\bar\t)},\\
&\int_{\bbT}{\log(1-|R_+(\t)|^2)}dm(\t)>-\infty;
\end{split}
\end{equation}
\item All other coefficients are of the form
\begin{equation}\label{2.10d}
T_-:=\frac{O}B,\
T_+(\tau)=\overline{T_-(\bar\tau)}\quad\text{and}\quad R_-:=-\frac
{T_-}{\bar T_+}\bar R_+,
\end{equation}
where $O$ is the outer function in the unit disk $\bbD$,  such that
\begin{equation}\label{8.9d}
|O|^2+|R_+(\t)|^2=1\ \text{a.e. on}\ \bbT
\end{equation}
and
\begin{equation}\label{normob}
    T_\mp(\k)=-ie^{ic_\pm}|T_\mp(\k)|.
\end{equation}
\end{itemize}
Also
\begin{equation}\label{3.10d}
\frac 1{\nu_+(\z_k)}\frac 1{\nu_-(\z_k)}= \left|\left(\frac
1{T_\pm}\right)'(\z_k)\right|^2.
\end{equation}

\end{proposition}

Now, we define the Faddeev--Marchenko Hilbert space.

\begin{definition}\label{deffm} Set
\begin{equation}\label{4.9d}
\alpha_+:=\{R_+,\nu_+\}.
\end{equation}
An element $f$ of the space $L^2_{\alpha_+}$ is a function on
$\bbT\cup \cZ$ such that
\begin{equation}\label{7.9d}
\begin{split}
||f||^2_{\alpha_+}= &\sum_{\z_k\in\cZ}|f(\z_k)|^2\nu_+(\z_k)\\
 +&\frac 1 2
\int_{\bbT}\begin{bmatrix} \overline{f(\t)}& \t\overline{f(\bar\t)}
\end{bmatrix}       \begin{bmatrix}         1&
         \overline{R_+(\t)}\\
      R_+(\t)& 1
            \end{bmatrix}
\begin{bmatrix}
          {f(\t)}\\ {\bar \t f(\bar\t)}
               \end{bmatrix}dm
\end{split}
\end{equation}
is finite.
\end{definition}

\begin{theorem}\label{thscrepr1}
For $\fA\in \fae$ the system
\begin{equation}\label{bofe}
    \{e^+(n,\z)\}_{n=-\infty}^{\infty}
\end{equation}
forms an orthonormal basis in the associated space $L^2_{\alpha_+}$.
Therefore, the map
\begin{equation}\label{screprp}
\cF^+:l^2(\bbZ)\to L^2_{\alpha_+}\quad \text{such that}\quad
\cF^+|n\rangle:=e^+(n,\z)
\end{equation}
is unitary. Moreover $\cF^+\fA(\cF^+)^*$ is the multiplication
operator by $v$.
\end{theorem}

\eqref{screprp} is called the scattering representation of $\fA$.
Note that simultaneously we have the representation
\begin{equation}\label{screprm}
\cF^-:l^2(\bbZ)\to L^2_{\alpha_-}\quad \text{such that}\quad
\cF^-|-n-1\rangle:=e^-(n,\z).
\end{equation}

\begin{theorem}\label{thscrepr2}
The scattering representations \eqref{screprp},
 \eqref{screprm} determine  each other by
\begin{equation}\label{splitf}
\begin{split}
    T_\pm(\tau)(\cF^\pm\tilde f)(\tau)=&\bar\tau(\cF^\mp\tilde f)(\bar\tau)
+R_\mp(\tau)(\cF^\mp\tilde f)(\tau),\quad\tau\in\bbT,\\
(\cF^{\pm}\tilde f)(\z_k)=&-\left(\frac 1 {T_\pm}\right)'(\z_k)
{\nu_\mp(\z_k)}(\cF^{\mp}\tilde f)(\z_k),\quad\z_k\in \cZ,
\end{split}
\end{equation}
for $\tilde f\in l^2(\bbZ)$, and have the following analytic
properties
\begin{equation}\label{anpr}
      (BT_\pm)\cF^\pm (l^2(\bbZ_\pm))\subset H^2.
\end{equation}
\end{theorem}

Finally, let us mention the important Wronskian identity. Put
formally
\begin{equation}\label{reprapeven}
\begin{split}
e^+(2m,\tau)&=e^{ic_+}b_\k^m b_{\bar\k}^m L_{\bar\k}(2m,\t),\\
\rho_{2m}e^+(2m+1,\tau)+ \bar a_{2m}e^+(2m,\tau)&=b_\k^m
b_{\bar\k}^{m} L_{\k}(2m,\t),
\end{split}
\end{equation}
and
\begin{equation}\label{reprapodd}
\begin{split}
e^+(2m+1,\tau)&=b_\k^m b_{\bar\k}^{m+1} L_{\k}(2m+1,\t),\\
\rho_{2m+1}e^+(2m+2,\tau)+ a_{2m+1}e^+(2m+1,\tau)&=e^{ic_+}b_\k^{m}
b_{\bar\k}^{m+1} L_{\bar\k}(2m+1,\t).
\end{split}
\end{equation}
Then
\begin{equation}\label{wids}
    \left|\begin{matrix}
\bar\t L_{\bar\k}(n,\bar\t)& \bar\t L_{\k}(n,\bar\t)\\
L_{\bar\k}(n,\t)&L_{\k}(n,\t)
\end{matrix}
\right|=\frac{d\log v(\tau)}{d\t}.
\end{equation}

\subsection{Inverse scattering: a brief discussion}
    The unimodular constant $e^{ic_+}$ and
the pair  $\alpha_+$ \eqref{4.9d} are called the  scattering data.

A fundamental question is how to recover the CMV matrix from the
scattering data? When can this be done? Do we have a uniqueness
theorem?

We say that the scattering data  are in the Szeg\"o over Blaschke
class, $\alpha_+\in \ale$, if
\begin{itemize}
\item $R_+$ is of the form \eqref{5.9d},
\item $\nu_+$ is a discrete measure supported on $\cZ$
\eqref{matbl}.
\end{itemize}
Let us point out that we did not even assume that the measure
$\nu_+$ is finite.

In short: to every scattering data of this class we can associate
the system of refection/trans\-mission coefficients by
\eqref{2.10d}, \eqref{8.9d}, \eqref{normob}, the dual measure
$\nu_-$ \eqref{3.10d} and the constant $e^{ic_-}$ \eqref{normob} in
such a way that there exists a CMV matrix from $\fae$, which
satisfies Theorem \ref{thszego} and Corollary \ref{corszego} with
these data.

To this end we associate with $\alpha_+$ the Faddeev--Marchenko
space $L^2_{\alpha_+}$, define  a Hardy type subspace
$\check{H}^2_{\alpha_+}$ in it, and, similar to \eqref{ts4bis},
construct the orthonormal basis (at this place the constant
$e^{ic_+}$ is required). Then, the multiplication operator (with
respect to this basis) is the CMV matrix and the claim of Theorem
\ref{thszego} is a Szeg\"o kind result on the asymptotics of this
orthonormal system.

For a brief explanation of the uniqueness problem we would like to
use the following analogy.  For the measure
$d\mu=\fw(\tau)dm(\tau)$, with $\log\fw\in L^1$, we can define the
Hardy space $\check{H}^2_{\mu}$ as the closer of $H^{\infty}$ (or
polynomials) in $L^2_{\mu}$--sense. On the other hand, let us define
the outer function $\phi$ such that $|\phi|^2=\fw$ and then define
\begin{equation}\label{exh2}
\hat{H}^2_{\mu}:=\left\{f=\frac{g}{\phi}: g\in H^2\right\}.
\end{equation}
According to the Beurling Theorem \cite{Garnett} these two Hardy
spaces are the same. But in the Faddeev--Marchenko setting their
counterpats $\check{H}^2_{\alpha_+}$ and $\hat{H}^2_{\alpha_+}$ not
necessarily coincide. For the data $\alpha_+$ the uniqueness takes
place if and only if $\check{H}^2_{\alpha_+}=\hat{H}^2_{\alpha_+}$.

\subsection{Hardy subspaces in the Faddeev--Marchenko space. Duality}
Let $\alpha_+\in \ale$. Define $T_\pm$, the dual data: $e^{ic_-}$,
$\alpha_-:=\{R_-,\nu_-\}$, and set
\begin{equation}\label{4.10d}
\begin{split}
\begin{bmatrix}
T_+f^+\\ T_- f^-
\end{bmatrix}(\t)=&
\begin{bmatrix}
T_+&0\\ R_+&1
\end{bmatrix}(\t)
\begin{bmatrix}
f^+(\t)\\  \bar \t f^+(\bar\t)
\end{bmatrix}\\
=&
\begin{bmatrix}
1&R_-\\ 0&T_-
\end{bmatrix}(\t)
\begin{bmatrix}
\bar \t f^-(\bar\t)\\  f^-(\t)
\end{bmatrix}
\end{split}
\end{equation}
for $\t\in\bbT$ and
\begin{equation}\label{5.10d}
f^{-}(\z_k)=-\left(\frac 1 {T_-}\right)'(\z_k)
{\nu_+(\z_k)}{f^+(\z_k)}
\end{equation}
for $\z_k\in\cZ$. It is evident that in this way we define a unitary
map from $L^2_{\alpha_+}$ to $L^2_{\alpha_-}$, in fact, due to
\eqref{4.10d}
\begin{equation}\label{1.22d}
\frac 1 {2}
               \int_{\bbT}\begin{bmatrix}
               \overline{f^+(\t)}&
               \overline{\bar \t f^+(\bar\t)}
               \end{bmatrix}
               \begin{bmatrix}
               1&
               \overline{R_+(\t)}\\
               R_+(\t)& 1
               \end{bmatrix}
               \begin{bmatrix}
               {f^+(\t)}\\
               {\bar \t f^+(\bar\t)}
               \end{bmatrix}dm=
               \frac{||T_+f^+||^2+||T_-f^-||^2}2,
\end{equation}
where in the RHS we have the standard $L^2$ norm on $\bbT$. The key
point is duality not only between these two spaces but, what is more
important, between corresponding Hardy subspaces.

Actually now we give two versions of  definitions of Hardy subspaces
(in general, {\it they are not equivalent!}). Due to the first one
$\check{H}^2_{\alpha_+}$ basically is the closer of $H^\infty$ with
respect to the given norm. More precisely, let $\mcB=\{B_N\}$, where
$B_N$ is a divisor of $B$ such that $B/B_N$ is a finite Blaschke
product. Then
\begin{equation}\label{10.9d}
f:=B_N g,\quad g\in H^\infty,\ B_N\in\mcB,
\end{equation}
belongs to $L^2_{\alpha_+}$. By $\check H^2_{\alpha_+}$ we denote
the closer in $L^2_{\alpha_+}$ of functions of the form
\eqref{10.9d}. Let us point out that every element $f$ of $\check
H^2_{\alpha_+}$ is such that $Of$ belongs to the standard $H^2$, see
\eqref{1.22d}. Therefore, in fact, $f(\z)$ has an analytic
continuation from $\bbT$ in the disk $\bbD$. Moreover, the value of
$f$ at $\z_k$, due to this continuation, and $f(\z_k)$ that should
be defined for all $\z_k\in\cZ$, since $f$ is a function from
$L^2_{\alpha_+}$, still perfectly coincide.

The second space also consists of functions from $L^2_{\alpha_+}$
having an analytic continuation in $\bbD$.
\begin{definition} A function $f\in L^2_{\alpha_+}$ belongs to
$\hat H^2_{\alpha_+}$ if $g(\t):=(BT_+ f)(\t)$, $\t\in \bbD$,
belongs to the standard $H^2$ and
$$
f(\z_k)=\left(\frac{g}{BT_+}\right)(\z_k), \ \z_k\in\cZ,
$$
where in the RHS $g$ and $BT_+$ are defined by  their analytic
continuation in $\bbD$.
\end{definition}

The following theorem  clarifies  relations between two Hardy
spaces.
\begin{theorem}\label{t1.3}
Let $f^+\in L^2_{\alpha_+}\ominus \check H^2_{\alpha_+}$ and let
$f^-\in L^2_{\alpha_-}$ be defined by \eqref{4.10d}, \eqref{5.10d}.
Then $f^-\in \hat H^2_{\alpha_-}$. In short, we write
\begin{equation}\label{0.17d}
(\hat H^2_{\alpha_-})^+= L^2_{\alpha_+}\ominus \check
H^2_{\alpha_+}.
\end{equation}
\end{theorem}

\subsection{Main results}
Both $\check H^2_{\alpha_+}$ and $\hat H^2_{\alpha_+}$ are spaces of
analytic functions in $\bbD$ with the reproducing kernels, which we
denote by $\check k_{\alpha_+,\z_0}=\check k_{\alpha_+}(\z,\z_0)$
and $\hat k_{\alpha_+,\z_0}=\hat k_{\alpha_+}(\z,\z_0)$
respectively. We put
\begin{equation}\label{scrkp}
    \check K_{\alpha_+,\z_0}=\frac{\check k_{\alpha_+,\z_0}}{\Vert \check
k_{\alpha_+,\z_0}\Vert},\quad \hat K_{\alpha_+,\z_0}=\frac{\hat
k_{\alpha_+,\z_0}}{\Vert \hat k_{\alpha_+,\z_0}\Vert}.
\end{equation}
Define the following shift operation on the scattering data
\begin{equation}\label{shift}
   \alpha_+^n= \{R^{(n)}_+,\nu_+^{(n)}\}:=
   \{b^{n}_{\k} b^{n}_{\bar\k}R_+,b^{n}_{\k} b^{n}_{\bar\k}\nu_+\}, \ n\in\bbZ.
\end{equation}
\begin{theorem}\label{thm1} Let $K_{\alpha+,\z_0}$ denote one of the
normalized kernel in \eqref{scrkp}. The system of functions
\begin{equation}\label{ts04bis}
e^+(n,\tau)=\begin{cases} b_\k^{m}
b_{\bar\k}^{m}K_{\alpha^n_+,\bar\k}(\t)e^{ic_+} ,
&n=2m\\
b_\k^{m} b_{\bar\k}^{m+1}K_{\alpha^n_+,\k}(\t) ,&n=2m+1
\end{cases}
\end{equation}
forms orthonormal basis in $\check H^2_{\alpha_+}$ and $\hat
H^2_{\alpha_+}$ respectively, if $n\in \bbZ_+$ and in the whole
$L^2_{\alpha_+}$ if $n\in \bbZ$. With respect to this basis the
multiplication operator by $v(\t)$ is the CMV matrix $\fA\in \fae$
with the coefficients given by \eqref{evp}. Moreover, the scattering
data given by Proposition \ref{propscat} and the dual orthonormal
system
\begin{equation}\label{emdef}
    T_-(\tau) e^-(-1-n,\tau):=\bar\tau
e^+(n,\bar\tau)+R_+(\t)e^+(n,\tau)
\end{equation}
correspond to $\fA$ in the sense of Theorem \ref{thszego} and
Corollary \ref{corszego}.
\end{theorem}

An important observation is the following
\begin{proposition}\label{prop120}
Let $\fA\in \fae$, and let $\alpha_+$ and $\cF^+$ correspond to this
matrix. Then
\begin{equation}\label{h}
    \check{H}^2_{\alpha_+}\subset\cF^+(l^2(\bbZ_+))\subset\hat{H}^2_{\alpha_+}.
\end{equation}
\end{proposition}
\noindent
Due to this observation Theorem \ref{thszego} is proved as
a corollary of Theorem \ref{thm1}.

\smallskip
 Concerning the uniqueness problem:
\begin{theorem}\label{thuni}
The scattering data $\alpha_+$, $e^{ic_+}$ determine $\fA\in \fae$
if and only if
\begin{equation}\label{uniccond}
\check k_{\alpha_{\pm}}(\k,\k) \check
k_{\alpha_{\mp}^{-1}}(\bar\k,\bar\k)
=\frac{1}{|T_\pm(\k)|^2}\frac{1}{(1-|\k|^2)^2}.
\end{equation}
\end{theorem}

\begin{corollary} Let $\fA\in \fae$ and $W$ be its spectral density
\eqref{dencity}. If
\begin{equation}\label{unsuff}
    \int_E W^{-1}(t)\,dm(t)<\infty,
\end{equation}
then there is no other CMV matrix of $\fae$ class corresponding to
the same scattering data.
\end{corollary}

In fact, \eqref{unsuff} means that $e^{\pm}(n,\t)\in L^2$ for
$n=-1,0$ and, therefore, for all $n\in\bbZ$. In this case there
exist the decompositions
\begin{equation}\label{transferpm}
    e^{\pm}(n,\t)=\sum_{l\ge n}M^{\pm}_{l,n}\fe_{l,c_\pm}(\t).
\end{equation}
The following matrix
\begin{equation}\label{trop}
    \cM_+=\begin{bmatrix}M^+_{0,0}&0&0&\dots\\
M^+_{1,0}&M^+_{1,1}&0&\dots\\
M^+_{2,0}&M^+_{2,1}&M^+_{2,2}&\dots\\
\vdots&\vdots&\vdots&\ddots
\end{bmatrix}
\end{equation}
yields the transformation (Gelfand--Levitan--Marchenko) operator,
acting in $l^2(\bbZ_{+})$. Similarly we define
$\cM_-:l^2(\bbZ_{-})\to l^2(\bbZ_{-})$ (for details see Sect.
\ref{shtr}). Note that under condition \eqref{unsuff} they are not
necessary bounded. We present necessary and sufficient conditions
when the scattering data determine the CMV matrix and both
transformation operators $\cM_{\pm}$ are bounded.

For $\theta\in\thte$  consider the following two conditions:

\begin{itemize}
\item[(i)] for all arcs $I\subset E$
\begin{equation}
\label{A2_220} \sup_{I} \langle \fw\rangle_I \langle
\fw^{-1}\rangle_I\,<\infty,
\end{equation}
where $\fw(t):=\frac{1-|\theta(t)|^2}{|1-\theta(t)|^2}$, and
$$\langle
\fw\rangle_I := \frac{1}{|I|}\int_I \fw(t)\, dm(t).
$$

\item[(ii)] for all arcs of the form $I=(e^{i\xi},e^{i\xi_0})$
or $I=(e^{-i\xi_0},e^{-i\xi})$, $I\subset\bbT\setminus E$,
\begin{equation}\label{NS20}
\sup_{I}\left\{\sum_{e^{i\eta_k}\in Y\cap I} \frac
{1}{\sqrt{|I||I_k|}\langle \fw\rangle_
 {I_k}}\frac{d\log v}{d\log\theta}(e^{i\eta_k})\right\}
\,<\infty,
\end{equation}
where $Y=\{e^{i\eta_k}\in\bbT\setminus E:\theta(e^{i\eta_k})=1\}$,
and
\begin{equation*}
    I_k=\begin{cases}(e^{i\xi_0}, e^{i(2\xi_0-\eta_k)}),&\eta_k>0,\\
(e^{-i(2\xi_0-\eta_k)},e^{-i\xi_0}),& \eta_k<0.
\end{cases}
\end{equation*}

\end{itemize}

\begin{theorem}\label{th23}
Let $\fA\in\fae$ with the associated Schur functions $\theta_\pm$
and scattering data $\alpha_+=\{R_+,\nu_+\}$. Then the following
statements are equivalent.
\begin{itemize}
\item[1.] The Schur functions $\theta_\pm$ satisfy (i), (ii).

\item[2.] The scattering data $\alpha_+$ determine a CMV matrix of
$\fae$ class uniquely and both related transformation operators are
bounded.
\end{itemize}
\end{theorem}

In Sect. \ref{secsc} we propose the following sufficient condition
given directly in terms of the scattering data.

With $\nu_+$ we associate the measure $\tilde \nu_+$ by
\begin{equation}\label{sscc1}
    \tilde \nu_+(\z_k)=\frac{1}{|B'(\z_k)|^2\nu_+(\z_k)}
\end{equation}
and with the reflection coefficient $R_+$ the Szeg\"o function
\begin{equation}\label{sscc2}
    \tilde R_+(\t)={R_+(\t)}{B(\t)^2}.
\end{equation}
\begin{theorem}\label{thsc}
Let $\tilde\nu_+$ be a Carleson measure in $\bbD$ and $\tilde R_+$
satisfy the following modification of the $A_2$ condition
\begin{equation}\label{sscc3}
\sup_{I}\frac{1}{|I|}\int_{I} \frac{|\tilde R_+-\Mst|^2+
(1-|\Mst|^2)}{1-|\tilde R_+|^2} \,dm<\infty.
\end{equation}
Then the data $\alpha_+=\{R_+,\nu_+\}$ determine the CMV matrix
uniquely for any $e^{ic_+}$. Moreover, the both GLM transformation
operators are bounded.
\end{theorem}

It shows that the class of data, comparably with the classical
Faddeev--Marchenko one, is indeed widely extended (an infinite set
of mass points and the reflection coefficient is very far necessary
to be a continuous function).
\begin{remark}
As we clarified in a discussion with A. Kheifets, in fact, our
condition is optimal in the class of conditions on the scattering
data with the following two properties: a) the condition is stable
with respect to the involution $R_+(\t)\mapsto -R_+(\t)$; b) the
assumptions on $R_+$ and $\nu_+$ are independent.
\end{remark}

\section{Proof of the Duality Theorem}

 The main goal of the  Lemma below is to clarify notations
that are, probably, a bit confusing. We believe that the diagram,
given in it, and the proof will help to avoid misunderstanding:
$\pm$-mappings $ L^2_{\alpha_+}\stackrel{\pm}{\longleftrightarrow}
L^2_{\alpha_-}$, defined by \eqref{4.10d}, \eqref{5.10d}, actually
depend of the scattering data $\{R_\pm,\nu_\pm\}$, although we do
not indicate this dependence explicitly.

\begin{lemma}\label{l1.4} Let $w(\z)$ be an inner
meromorphic function in $\bbD$ such that $w(\z_k)\not=0$,
$w(\z_k)\not=\infty$ for all $\z_k\in\cZ$. Put
$w_*(\z):=\overline{w(\bar\z)}$. The following diagram is
commutative
\begin{equation}\label{18j31}
\begin{array}{lll}
L^2_{\{w w_* R_+,w w_*\nu_+\}}
&\stackrel{w}{\longrightarrow}& L^2_{\alpha_+}\\
\llap{+}\Big{\uparrow}\Big\downarrow\rlap{-} & &
\llap{+}\Big{\uparrow}\Big\downarrow\rlap{-}\\
L^2_{\{w^{-1} w_*^{-1} R_-,w^{-1}
w_*^{-1}\nu_-\}}&\stackrel{w_*^{-1}}{\longrightarrow}&
L^2_{\alpha_-}
\end{array}
\end{equation}
Here the horizontal arrows are related to the unitary multiplication
operators and the vertical arrows are related to two
\textsl{different} $\pm$--duality mappings.
\end{lemma}

\begin{proof}
Note that both $w$ and $w_*^{-1}$ are well defined on $\bbT\cup\cZ$.
Evidently, $wf\in L^2_{\alpha_+}$ means that $f\in L^2_{\{w w_*
R_+,w w_*\nu_+\}}$. Also, since $|w(\t)|=1$, $\t\in \bbT$, we have
that $\{w^{-1} w_*^{-1} R_-,w^{-1}w_*^{-1}\nu_-\}$ are
minus--scattering data for $\{w w_* R_+,w w_*\nu_+\}$ if $\alpha_-$
corresponds to $\alpha_+$. In other words $T_\pm$--functions remain
the same for both sets of scattering data. Then we use definitions
\eqref{4.10d}, \eqref{5.10d}.
\end{proof}

\begin{proof}[Proof of Theorem \ref{t1.3}]
Let us mention that $f^+\in L^2_{\alpha_+}$ implies
$$
(T_-f^-)(\t)=R_+(\t)f^+(\t)+\bar\t f^+(\bar\t)\in L^2, \ \t\in\bbT.
$$
Since
$$
\langle f^+, Bh \rangle_{\alpha_+}= \langle R_+(\t)f^+(\t)+\bar \t
f^+(\bar\t), \bar\t B(\bar\t)h(\bar\t) \rangle,\ h\in H^2,
$$
it follows from $f^+\in L^2_{\alpha_+}\ominus \check H^2_{\alpha_+}$
that
$$
(BT_-f^{-})(\t)=g(\t) :=B(\t)(R_+(\t)f^+(\t)+\bar\t f^+(\bar\t))\in
H^2.
$$

Now we calculate the scalar product
\begin{equation*}
\begin{split}
\langle f^+, \frac{B(\tau)}{\t-\z_k} \rangle_{\alpha_+}=&
f^+(\z_k)\overline{B'(\z_k)}\nu_+(\z_k) +\langle BT_-
f^-,\frac{1}{1-\t\bar\z_k} \rangle\\=& f^+(\z_k)B'(\z_k)\nu_+(\z_k)
+ g(\z_k)=0.
\end{split}
\end{equation*}
Therefore, by \eqref{5.10d} we get
$$
f^-(\z_k)=\left(\frac{g}{BT_-}\right)(\z_k), \ \z_k\in\cZ.
$$

For the converse direction we calculate the scalar product of
$f^+\in \hat{H}^2_{\alpha_+}$ with a function of the form $B_N g$,
$B_N\in\mathcal B$, $g\in H^2$ and use the fact that $BT_-f^-\in
H^2$.
\end{proof}

\section{Reproducing Kernels}

We prove several propositions concerning specific properties of the
reproducing kernels in $\check H^2_{\alpha_+}$ and $\hat
H^2_{\alpha_+}$. The multiplication operator by $v$ is playing an
essential role in these constructions.

\begin{lemma}\label{l2.1f5}
Let $\check k_{\alpha_+}(\z,\k)$ and $\hat k_{\alpha_+}(\z,\k)$
 denote the reproducing kernels of the spaces
$\check H^2_{\alpha_+}$ and $\hat H^2_{\alpha_+}$ respectively. Then
\begin{equation}\label{2.1f1}
(\check k_{\alpha_+}(\z,\k))^-= \frac{1-\z\k}{(\z-\bar
\k)(1-|\k|^2)}\frac{1}{T_-(\bar\k)} \frac{\hat
k_{\alpha_-^{-1}}(\z,\bar\k)} {\hat
k_{\alpha_-^{-1}}(\bar\k,\bar\k)},
\end{equation}
and, therefore,
\begin{equation}\label{2.1af1}
\check k_{\alpha_+}(\k,\k) \hat k_{\alpha_-^{-1}}(\bar\k,\bar\k)
=\frac{1}{|T_{-}(\bar\k)|^2}\frac{1}{(1-|\k|^2)^2}.
\end{equation}
\end{lemma}

\begin{proof} First we note that the following one--dimensional
spaces coincide:
\begin{equation*}
\{(\check k_{\alpha_+}(\z,\k)\}^-= \{b^{-1}_{\bar\k} \hat
k_{\alpha_-^{-1}}(\z,\bar\k)\}.
\end{equation*}
It follows immediately from Theorem \ref{t1.3}, but let us give a
formal prove. Starting with the orthogonal decomposition
\begin{equation*}
\{\check k_{\alpha_+}(\z,\k)\}= \check H^2_{\alpha_+}\ominus
b_{\k}\check H^2_{\alpha_+^{1}}
\end{equation*}
we have
\begin{equation*}
\{\check k_{\alpha_+}(\z,\k)\}^-= (\check H^2_{\alpha_+})^-\ominus
(b_{\k}\check H^2_{\alpha_+^{1}})^-,
\end{equation*}
or, due to \eqref{18j31},
\begin{equation*}
\{\check k_{\alpha_+}(\z,\k)\}^-= (\check H^2_{\alpha_+})^-\ominus
b^{-1}_{\bar\k}(\check H^2_{\alpha_+^{1}})^-.
\end{equation*}

Now we use Theorem \ref{t1.3}
\begin{equation*}
\begin{split}
\{k_{\alpha_+}(\z,\k)\}^-=& (L^2_{\alpha_-}\ominus \hat
H^2_{\alpha_-}) \ominus b^{-1}_{\bar\k}(L^2_{\alpha_-^{-1}} \ominus
\hat
H^2_{\alpha_-^{-1}})\\
=&b^{-1}_{\bar\k} (\hat H^2_{\alpha_-^{-1}} \ominus b_{\bar\k}\hat
H^2_{\alpha_-}).
\end{split}
\end{equation*}
Thus
\begin{equation}\label{2.2f1}
(k_{\alpha_+}(\z,\k))^-= C b^{-1}_{\bar\k} \hat
k_{\alpha_-^{-1}}(\z,\bar\k).
\end{equation}

The essential part of the lemma deals with the constant $C$. We
calculate the scalar product
\begin{equation*}
\left\langle \check k_{\alpha_+}(\t,\k), \frac{B}{1-\t\bar\k}
\right\rangle_{\alpha_+}.
\end{equation*}
On the one hand, since $\frac{B}{1-\z\bar\k}$ belongs to the
intersection of $L^2_{\alpha_+}$ with $H^2$, we can use the
reproducing property of $\check k_{\alpha_+}$:
\begin{equation}\label{2.3f1}
\left\langle \check k_{\alpha_+}(\t,\k), \frac{B}{1-\t\bar\k}
\right\rangle_{\alpha_+} = \overline{\frac{B(\k)}{1-|\k|^2}}=
\frac{B(\bar\k)}{1-|\k|^2}.
\end{equation}

On the other hand we can reduce the given scalar product to the
scalar product in the standard $H^2$. Since $B(\z_k)=0$, the
$\nu$--component vanishes and we get
\begin{equation*}
\begin{split}
\frac 1 2 &\left\langle
\begin{bmatrix} 1&\bar R_+\\
R_+&1\end{bmatrix}(\t)
\begin{bmatrix}
\check k_{\alpha_+}(\t,\k)\\ \bar \t \check k_{\alpha_+}(\bar\t,\k)
\end{bmatrix},
\begin{bmatrix}
\frac{B(\t)}{1-\t\bar\k}\\
\frac{B(\bar\t)}{t-\bar\k}
\end{bmatrix}
\right\rangle\\=& \left\langle T_-(\t) (\check
k_{\alpha_+}(\t,\k))^-, \frac{\bar{B}}{\t-\bar\k} \right\rangle.
\end{split}
\end{equation*}
Substituting here \eqref{2.2f1} we get
\begin{equation*}
C\left\langle (BT_-)(\t) \hat k_{\alpha_-^{-1}}(\t,\bar\k),
b_{\bar\k}(\t)\frac{1}{\t-\bar\k} \right\rangle.
\end{equation*}
Since $(BT_-)(\z) \hat k_{\alpha_-^{-1}}(\z,\bar\k)$ belongs to
$H^2$ and $b_{\bar\k}(\z)\frac{1}{\z-\bar\k}=
e^{ic}\frac{1}{1-\z\k}$ is collinear to the reproducing kernel here,
we get recalling \eqref{2.3f1}
\begin{equation*}
e^{-ic}C (BT_-)(\bar\k) \hat k_{\alpha_-^{-1}}(\bar\k,\bar\k)=
\frac{B(\bar\k)}{1-|\k|^2}.
\end{equation*}

Thus \eqref{2.1f1} is proved. Comparing the norms of that vectors
and taking into account that the $-$--map is an isometry we get
\eqref{2.1af1}.
\end{proof}


Consider the multiplication operator by $\bar v$, acting in
\begin{equation}\label{2.16d}
L^2_{\alpha_+}=(\hat H^2_{\alpha_-})^+\oplus
 \check H^2_{\alpha_+}.
\end{equation}
\begin{lemma}\label{l2.2f5}
The multiplication operator by $\bar v$ acts as a unitary operator
from
\begin{equation}\label{5.16d}
\{\hat k_{\alpha_-}^+(\z,\k)\}\oplus \check H^2_{\alpha_+}
\end{equation}
to
\begin{equation}\label{6.16d}
\{\hat k_{\alpha_-}^+(\z,\bar\k)\}\oplus \check H^2_{\alpha_+}.
\end{equation}
\end{lemma}

\begin{proof}
It is evident that the multiplication by $\bar v=\frac{b_{\bar\k}}
{b_{\k}}$ acts from
\begin{equation*}
\{f\in \hat H^2_{\alpha_-}: f(\k)=0\}= b_{\k}\hat H^2_{\alpha_-^1}
\end{equation*}
to
\begin{equation*}
\{f\in \hat H^2_{\alpha_-}: f(\bar\k)=0\} = b_{\bar\k}\hat
H^2_{\alpha_-^1}.
\end{equation*}
Therefore it acts in their orthogonal complements \eqref{5.16d},
\eqref{6.16d}.

\end{proof}

Recall definition of the characteristic function of a unitary node
and its functional model. Let $U$ be a unitary operator acting from
$K\oplus E_1$ to $K\oplus E_2$. We assume that the Hilbert spaces
 $E_1$ and $E_2$ are finite--dimensional
 (actually, in this section we need $\dim E_1=\dim E_2=1$).
 The characteristic function is defined by
\begin{equation}\label{7.16d}
\Theta(w):=P_{E_2}U(I_{K\oplus E_1}-w P_K U)^{-1}|E_1.
\end{equation}
It is a holomorphic in the unit disk contractive--valued operator
function. We make a specific assumption that $\Theta(w)$ has an
analytic continuation in the exterior of the unite disk through a
certain arc $(a,b)\subset\bbT$ due to the symmetry principle:
$$
\Theta(w)=\Theta^*\left(\frac 1{\bar w}\right)^{-1}.
$$

For $f\in K$ define
\begin{equation}\label{2.5f4}
F(w):=P_{E_2}U(I-w P_K U)^{-1}f.
\end{equation}
This $E_2$--valued holomorphic vector function belongs to the
functional space $K_{\Theta}$ with the following properties.

\begin{itemize}
\item $F(w)\in H^2(E_2)$, moreover it has analytic continuation
through the arc $(a,b)$.
\item $F_*(w):=\Theta^*(w)F\left(\frac 1{\bar w}\right)
\in H^2_-(E_1)$.
\item For almost every $w\in \bbT$ the vector
$\begin{bmatrix}F_*\\ F\end{bmatrix}(w)$ belongs to the image of the
operator $\begin{bmatrix} I&\Theta^*\\ \Theta & I\end{bmatrix}(w)$,
and therefore the scalar product
$$
\left\langle
\begin{bmatrix} I&\Theta^*\\ \Theta &
I\end{bmatrix}^{[-1]}\begin{bmatrix}F_*\\ F\end{bmatrix},
\begin{bmatrix}F_*\\ F\end{bmatrix}
\right\rangle_{E_1\oplus E_2}
$$
has sense and does not depend of the choice of a preimage (the first
term in the above scalar product). Moreover
\begin{equation}\label{2.6f4}
\int_{\bbT} \left\langle
\begin{bmatrix} I&\Theta^*\\ \Theta &
I\end{bmatrix}^{[-1]}\begin{bmatrix}F_*\\ F\end{bmatrix},
\begin{bmatrix}F_*\\ F\end{bmatrix}
\right\rangle_{E_1\oplus E_2} dm<\infty.
\end{equation}
\end{itemize}
The integral in \eqref{2.6f4} represents the square of the norm of
$F$ in $K_{\Theta}$.

Note that in the model space $P_K U\vert K$ became a certain
"standard" operator
\begin{equation}\label{2.11f23}
f\mapsto F(w)\quad\Longrightarrow\quad P_K
Uf\mapsto\frac{F(w)-F(0)}{w},
\end{equation}
see \eqref{2.5f4}.

The following simple identity is a convenient tool in the
forthcoming calculation.
\begin{lemma} For a unitary $U:K\oplus E_1\to
K\oplus E_2$
\begin{equation}\label{1.17d}
U^*P_{E_2}U(I-w P_K U)^{-1}= I+(w-U^*)P_{K}U(I-w P_K U)^{-1}.
\end{equation}
\end{lemma}
\begin{proof} Since $I_{K\oplus E_2}=P_K+P_{E_2}$
and $U$ is unitary we have
\begin{equation*}
U^*P_{E_2}U= (I-w P_K U)+(w-U^*)P_{K}U.
\end{equation*}
Then we multiply this identity by $(I-w P_K U)^{-1}$.
\end{proof}

\begin{theorem}
Let $e_1$, $e_2$ be the normalized vectors of the one-dimensional
spaces \eqref{5.16d} and \eqref{6.16d}
\begin{equation}\label{1.20d}
\begin{split}
e_1(\z)=&\frac{1}{b_{\bar\k}}\frac{\check
k_{\alpha_+^{-1}}(\z,\bar\k)}{\sqrt{\check
k_{\alpha_+^{-1}}(\bar\k,\bar\k)}}
=-i\frac{T_+(\bar\k)}{|T_+(\bar\k)|}\frac{\hat
k_{\alpha_-}^+(\z,\k)}{\sqrt{\hat
k_{\alpha_-}(\k,\k)}},\\
e_2(\z)=&\frac{1}{b_{\k}}\frac{\check
k_{\alpha_+^{-1}}(\z,\k)}{\sqrt{\check k_{\alpha_+^{-1}}(\k,\k)}}=
i\frac{T_+(\k)}{|T_+(\k)|}\frac{ \hat k^+_{\alpha_-}(\z,\bar\k)}{
\sqrt{\hat k_{\alpha_-}(\bar\k,\bar\k)}}.
\end{split}
\end{equation}
Then the reproducing kernel of $\check H^2_{\alpha_+}$ is of the
form
\begin{equation}\label{2.13f6}
\check k_{\alpha_+}(\z,\z_0)= \frac{(v e_2)(\z)\overline{(v
e_2)(\z_0)} -e_1(\z)\overline{e_1(\z_0)}}{1- v (\z)
\overline{v(\z_0)}}.
\end{equation}
\end{theorem}

\begin{proof} First,
we are going to find the characteristic function of the
multiplication operator by $\bar v$ with respect to decompositions
\eqref{5.16d} and \eqref{6.16d} and the corresponding functional
representation of this node.

By \eqref{1.20d} we fixed `basises' in the one-dimensional spaces.
So, instead of the operator we get the matrix, in fact the scalar
function $\theta(w)$:
\begin{equation}\label{3.20d}
\Theta(w) e_1:= P_{E_2}U (I-w P_K U)^{-1}e_1 =e_2\theta(w).
\end{equation}
Let us substitute \eqref{3.20d} into \eqref{1.17d}
\begin{equation}\label{2.14f5}
v(\z)e_2(\z) \theta(w) = e_1(\z) +(w-v(\z))(P_{K}U(I-w P_K
U)^{-1}e_1) (\z).
\end{equation}
Recall an important property of $\hat k_{\alpha_-}^+(\z,\k)$: it has
analytic continuation in the $\bbD$ with the only pole at $\bar\k$
(see Lemma \ref{l2.1f5}). Therefore all terms in \eqref{2.14f5} are
analytic in $\z$ and we can chose $\z$ such that $v(\z)=w$. Then we
obtain the characteristic function in terms of the reproducing
kernels
\begin{equation}\label{2.16f8}
\theta(v(\z)) =\frac {e_1(\z)}{v(\z)e_2(\z)}=\frac{\check
k_{\alpha_+^{-1}}(\z,\bar\k)}{\check k_{\alpha_+^{-1}}(\z,\k)}.
\end{equation}

Similarly for $f\in K=\check H^2_{\alpha_+}$ we define the scalar
function $F(w)$ by
\begin{equation}\label{2.17f8}
P_{E_2}U(I-w P_{K} U)^{-1}f =e_2 F(w).
\end{equation}
Using again \eqref{1.17d} we get
\begin{equation*}
v(\z)e_2(\z) F(w) = f(\z) +(w-v(\z))(P_{K}U(I-w P_K U)^{-1}f) (\z).
\end{equation*}
Therefore,
\begin{equation}\label{2.18f8}
F(v(\z)) =\frac {f(\z)}{v(\z)e_2(\z)}.
\end{equation}

Now we are in a position to get \eqref{2.13f6}. Indeed, by
\eqref{2.17f8} and \eqref{2.18f8} we proved that the vector
$$
P_{K}(I-\overline{v(\z_0)} U^*P_K)^{-1}U^* e_2
\overline{v(\z_0)e_2(\z_0)}
$$
is the reproducing kernel of $K=\check H^2_{\alpha_+}$ with respect
to $\z_0$, $|v(\z_0)|<1$. Using the Darboux identity
\begin{equation*}
P_{E_2}U(I-w P_{K} U)^{-1} P_{K}(I-{\bar w_0} U^*P_K)^{-1}U^*| E_2=
\frac{I-\Theta(w)\Theta^*(w_0)}{1-w\bar w_0}
\end{equation*}
(in the given setting it is a simple and pleasant  exercise) we
obtain
\begin{equation}\label{rc13}
\check k_{\alpha_+}(\z,\z_0)= v(\z)e_2(\z)
\frac{I-\theta(v(\z))\overline{\theta(v(\z_0))}}
{1-v(\z)\overline{v(\z_0)}} \overline{v(\z_0)e_2(\z_0)}
\end{equation}
for $|v(\z)|<1$, $|v(\z_0)|<1$. By \eqref{2.16f8} we have
\eqref{2.13f6} that, by analyticity, holds for all $|\z|<1$, $
|\z_0|<1$.
\end{proof}

\begin{corollary}
The following Wronskian--kind identity is satisfied for the
reproducing kernels
\begin{equation}\label{wif2}
\left|
\begin{matrix}
(T_- e_2^{-})(\z)&
 (T_- e_1^{-})(\z)\\
e_2(\z)& e_1(\z)
\end{matrix}\right |=
- (\log v(\z))',\ |\z|<1.
\end{equation}
\end{corollary}

\begin{proof}
We multiply $\check k^-_{\alpha_+}(\z,\bar \z_0)$ by $b_{\z_0}(\z)$
and calculate the resulting function of $\z$ at $\z=\z_0$. By
\eqref{2.1f1} we get
\begin{equation}\label{2.15f3}
\{b_{\z_0}(\z)\check k^-_{\alpha_+}(\z,\bar\z_0)\}_{\z=\z_0}
=e^{ic}\frac{1}{T_-(\z_0)(1-|\z_0|^2 )}.
\end{equation}
Now we make the same calculation but using  representation
\eqref{2.13f6}. We have
\begin{equation*}
\check k^-_{\alpha_+}(\z,\bar\z_0)\\
=\frac {-v(\z_0)} {v(\z)-v(\z_0)}
\left|
\begin{matrix}
v(\z) e_2^{-}(\z)&
 e_1^{-}(\z)\\
\overline{e_1(\bar\z_0)}& \overline{v(\bar\z_0)e_2(\bar\z_0)}
\end{matrix}\right |,
\end{equation*}
or, after multiplication by $b_{\z_0}(\z)$,
\begin{equation*}
\{b_{\z_0}(\z)\check k^-_{\alpha_+}(\z,\bar\z_0)\}_{\z=\z_0} =
e^{ic} \frac {-v(\z_0)} {v'(\z_0)(1-|\z_0|^2 )} \left|
\begin{matrix}
v(\z) e_2^{-}(\z)&
 e_1^{-}(\z)\\
\overline{e_1(\bar\z_0)}& \overline{v(\bar\z_0)e_2(\bar\z_0)}
\end{matrix}\right |.
\end{equation*}
In combination with \eqref{2.15f3}, we get
\begin{equation*}
 -\frac{v'(\z_0)}
{v(\z_0)T_-(\z_0)} = \left|
\begin{matrix}
v(\z_0) e_2^{-}(\z_0)&
 e_1^{-}(\z_0)\\
\overline{e_1(\bar\z_0)}& v^{-1}(\z_0)\overline{e_2(\bar\z_0)}
\end{matrix}\right |.
\end{equation*}

Due to the symmetry $\overline{\hat k_{\alpha_-}(\z,\z_0)} =\hat
k_{\alpha_-}(\bar\z,\bar\z_0)$, we have
$\overline{e_2(\bar\z_0)}=e_1(\z_0)$. Thus \eqref{wif2} is proved.
\end{proof}

\begin{corollary}
Let $\t\in \bbT$, then
\begin{equation}\label{wif8}
|e_2(\t)|^2-
 |e_1(\t)|^2=\frac
{d\log v(\t)}{d\log\t}.
\end{equation}
\end{corollary}
\begin{proof} All terms of \eqref{wif2} have boundary values.
Recall that
$$
(T_- e_{1,2}^-)(\t)=(R_+ e_{1,2})(\t)+\bar\t
e_{1,2}(\bar\t),\quad\t\in \bbT.
$$
Then use again the symmetry of the reproducing kernel.
\end{proof}

\section{A recurrence relation for reproducing kernels and the Schur parameters}

Let
\begin{equation}\label{rknorm}
    K_{\alpha}(\zeta,\zeta_0):=\frac{k_{\alpha}(\zeta,\zeta_0)}{\sqrt{k_{\alpha}(\zeta_0,\zeta_0)}},
\end{equation}
where $k_{\alpha}(\zeta,\zeta_0)$ denotes one of reproducing kernels
$\hat k_{\alpha_\pm}(\zeta,\zeta_0)$ or $\check
k_{\alpha_\pm}(\zeta,\zeta_0)$.

\begin{theorem}
Both systems
$$
\{K_{\alpha}(\zeta,\k),
b_{\k}(\zeta)K_{\alpha^{1}}(\zeta,\overline{\k})\}
$$
and
$$
\{K_{\alpha}(\zeta,\overline{\k}),
b_{\bar\k}(\zeta)K_{\alpha^{1}}(\zeta,{\k})\}
$$
form an orthonormal basis in the two dimensional space spanned by
$K_{\alpha}(\zeta,\k)$ and $K_{\alpha}(\zeta,\overline{\k})$.
Moreover
\begin{equation}\label{f2}
\begin{matrix}
K_{\alpha}(\zeta,\overline{\k})=& a(\alpha)K_{\alpha}(\zeta,{\k})
+\rho(\alpha) b_{\k}(\zeta)K_{\alpha^{1}}(\zeta,\overline{\k}),
\\
K_{\alpha}(\zeta,\k)=&\overline{a(\alpha)}
K_{\alpha}(\zeta,\overline{\k}) +\rho(\alpha)
b_{\bar\k}(\zeta)K_{\alpha^{1}}(\zeta,{\k}),
\end{matrix}
\end{equation}
where
\begin{equation}\label{def}
a(\alpha)=a=\frac{K_{\alpha}(\k,\overline{\k})}
{K_{\alpha}(\k,\k)},\quad \rho(\alpha)=\rho=\sqrt{1-|a|^2}.
\end{equation}
\end{theorem}

\begin{proof}
The first claim is evident, therefore
$$
K_{\alpha}(\zeta,\overline{\k})= c_1 K_{\alpha}(\zeta,{\k}) +c_2
b_{\k}(\zeta)K_{\alpha^{1}}(\zeta,\overline{\k}).
$$
Putting $\zeta=\k$ we get $c_1=a$. Due to orthogonality we have
$$
1=|a|^2+|c_2|^2.
$$
Now, put $\zeta=\overline{\k}$. Taking into account that
$K_{\alpha}(\k,\overline{\k})
=\overline{K_{\alpha}(\overline{\k},\k)}$ and that by normalization
$b_{\k}(\overline{\k})>0$ we prove that $c_2$ being positive is
equal to $\sqrt{1-|a|^2}$. Note that simultaneously we proved that
$$
\rho(\alpha)=b_{\k}(\overline{\k})
\frac{K_{\alpha^{1}}(\overline{\k},\overline{\k})}
{K_{\alpha}(\overline{\k},\overline{\k})}.
$$
\end{proof}

\begin{corollary} A recurrence relation for reproducing kernels
generated by the shift  of the scattering data  is of the form
\begin{equation}\label{rr}
\begin{split}
b_{\k}({\zeta})&
\begin{bmatrix}
K_{\alpha^{1}}(\zeta,{\k}), &-K_{\alpha^{1}}(\zeta,\overline{\k})
\end{bmatrix}
\\
=&
\begin{bmatrix}
K_{\alpha}(\zeta,{\k}), &-K_{\alpha}(\zeta,\overline{\k})
\end{bmatrix}
\frac 1\rho\begin{bmatrix}
1&a\\
\bar a&1
\end{bmatrix}
\begin{bmatrix}
v&0\\
0&1
\end{bmatrix}.
\end{split}
\end{equation}
\end{corollary}
\begin{proof}
Recalling $v=b_{\k}/b_{\bar\k}$, we write
\begin{equation*}
\begin{split}
b_{\k}({\zeta})&
\begin{bmatrix}
K_{\alpha^{1}}(\zeta,{\k}), &-K_{\alpha^{1}}(\zeta,\overline{\k})
\end{bmatrix}\\
=&
\begin{bmatrix}b_{\bar\k}({\zeta})
K_{\alpha^{1}}(\zeta,{\k}),
&-b_{\k}({\zeta})K_{\alpha^{1}}(\zeta,\overline{\k})
\end{bmatrix}
\begin{bmatrix}
v&0\\
0&1
\end{bmatrix}.
\end{split}
\end{equation*}
Then, use \eqref{f2}.
\end{proof}
\begin{corollary}\label{c3}
Let
\begin{equation}\label{sa}
\theta_\alpha(v):=\frac{K_{\alpha}(\zeta,\overline{\k})}
{K_{\alpha}(\zeta,{\k})}.
\end{equation}
Then the Schur parameters of the function $e^{ic} \theta_\alpha(v)$,
 are
$$
\{e^{ic} a(\alpha^n)\}_{n=0}^\infty.
$$
\end{corollary}

\begin{proof}
Let us note that \eqref{rr} implies
$$
\theta_\alpha(v)= \frac{a(\alpha)+v \theta_{\alpha^1}(v)}
{1+\overline{a(\alpha)}v \theta_{\alpha^1}(v)}
$$
and that $|a(\alpha)|<1$. Then we iterate this relation. Also,
multiplication by $e^{ic}\in \bbT$ of a Schur class function
evidently leads to multiplication by $e^{ic}$ of all Schur
parameters.
\end{proof}

\begin{theorem}
The multiplication operator  with respect to the basis
\eqref{ts04bis} is CMV.
\end{theorem}

\begin{proof}
Recall \eqref{blfackappa}, from which we can see that the
decomposition of the vector $v(\zeta)
K_{\alpha}(\zeta,\overline{\k})$ is of the form
$$
v(\zeta) K_{\alpha}(\zeta,\bar{\k})=
c_0\frac{K_{\alpha^{-2}}(\zeta,\bar{\k})}
{b_{\k}(\zeta)b_{\bar{\k}}(\zeta)}+
c_1\frac{K_{\alpha^{-1}}(\zeta,{\k})} {b_{\k}(\zeta)}+
c_2{K_{\alpha}(\zeta,\bar{\k})} +
c_3b_{\bar{\k}}(\zeta){K_{\alpha^1}(\zeta,{\k})}.
$$
 Multiplying by the denominator
${b_{\k}(\zeta)b_{\bar{\k}}(\zeta)}$ we get
\begin{equation}\label{9s11}
\begin{split}
b_{\k}^2(\zeta) K_{\alpha}(\zeta,\bar{\k})=&
c_0{K_{\alpha^{-2}}(\zeta,\bar{\k})} +
c_1{K_{\alpha^{-1}}(\zeta,{\k})} {b_{\bar{\k}}(\zeta)}\\+&
c_2{K_{\alpha}(\zeta,\bar{\k})} b_{\k}(\zeta)b_{\bar{\k}}(\zeta) +
c_3 {K_{\alpha^1}(\zeta,{\k})}b_{\k}(\zeta)b^2_{\bar{\k}}(\zeta).
\end{split}
\end{equation}
First we put $\zeta=\bar{\k}$. By the definition of $\rho(\alpha)$
we have
$$
c_0=b_{\k}^2(\bar{\k}) \frac{K_{\alpha}(\bar{\k},\bar{\k})}
{K_{\alpha^{-2}}(\bar{\k},\bar{\k})} =\rho(\alpha^{-1})
\rho(\alpha^{-2}).
$$
Putting $\zeta={\k}$ in \eqref{9s11} and using the definition of
$a(\alpha)$, we have
$$
c_1=-c_0\frac{K_{\alpha^{-2}}(\k,\bar{\k})}
{K_{\alpha^{-1}}(\k,{\k}) b_{\bar{\k}}(\k)}= -\rho(\alpha\mu)
\rho(\alpha^{-2}) \frac{a(\alpha^{-2})}{\rho(\alpha^{-2})}=
-\rho(\alpha^{-1}) a(\alpha^{-2}).
$$

Doing in the same way we can find a representation for $c_2$ that
would involve derivatives of the reproducing kernels. However, we
can find $c_2$ in terms of $a$ and $\rho$  calculating the scalar
product
$$
c_2=\langle b_{\k}^2(\zeta) K_{\alpha}(\zeta,\bar{\k}),
{b_{\k}(\zeta)b_{\bar\k}(\zeta)} {K_{\alpha}(\zeta,\bar{\k})}
\rangle.
$$
Since $b_{\k}(\zeta)$ is unimodular, using \eqref{f2}, we get
$$
c_2=\left\langle\frac{K_{\alpha^{-1}}(\zeta,\bar{\k}) -
a(\alpha^{-1})K_{\alpha^{-1}}(\zeta,\k)} {\rho(\alpha^{-1})},
{b_{\bar\k}(\zeta)} {K_{\alpha}(\zeta,\overline{\k})} \right\rangle.
$$
Recall that $ k_{\alpha}(\zeta,\k)=
K_{\alpha}(\zeta,\k)K_{\alpha}(\k,\k) $ is the reproducing kernel.
Thus
$$
c_2=-\frac{a(\alpha^{-1})}{\rho(\alpha^{-1})}\overline{
\frac{{b_{\bar\k}(\k)} {K_{\alpha}(\k,\overline{\k})}}
{K_{\alpha^{-1}}(\k,\k)}}= -\frac{a(\alpha^{-1})}{\rho(\alpha^{-1})}
\overline{\rho(\alpha^{-1})a(\alpha)}
=-a(\alpha^{-1})\overline{a(\alpha)}.
$$
And, similar,
$$
c_3=\left\langle\frac{K_{\alpha^{-1}}(\zeta,\bar{\k}) -
a(\alpha^{-1})K_{\alpha^{-1}}(\zeta,\k)} {\rho(\alpha^{-1})},
{b^2_{\bar\k}(\zeta)} {K_{\alpha^1}(\zeta,{\k})} \right\rangle.
$$
Thus
$$
c_3=-\frac{a(\alpha^{-1})}{\rho(\alpha^{-1})}\overline{
\frac{{b^2_{\bar\k}(\k)} {K_{\alpha^1}(\k,{\k})}}
{K_{\alpha^{-1}}(\k,\k)}}= -\frac{a(\alpha^{-1})}{\rho(\alpha^{-1})}
\overline{\rho(\alpha^{-1})\rho(\alpha)}
=-a(\alpha^{-1}){\rho(\alpha)}.
$$

To find the decomposition of the vector $v(\zeta)
\frac{K_{\alpha^{-1}}(\zeta,{\k})}{b_{\k}(\z)}$ is even simpler.
Note that all other columns of the CMV matrix,  starting from these
two, can be obtain by the two step shift  of the scattering data.

\end{proof}

\section{From the spectral data to the scattering data: a special representation of the
Schur function}\label{specshur}

In this section we use   Theorem D \cite{syu}, see also \cite{scf}.
For readers convenience we formulate it here.

\begin{theorem}\label{D}

Let $r(v)$ be a meromorphic function in $\Omega$ with the property
\begin{equation}\label{profr}
    \frac{r(v(\z))+\overline{r(v(\z))}}{i(\z-\bar\z)}\ge 0.
\end{equation}
If poles $\{t_j\}$ of $r(v)$ (they should lie on $\bbT\setminus E$
due to \eqref{profr}) satisfy the Blaschke condition \eqref{matbl},
then $r(v(\z))$ is of bounded characteristic in $\bbD$, and in
addition the inner (in the Beurling sense) factor of $r(v(\z))$ is a
quotient of Blaschke products, i.e., it does not have a singular
inner factor.
\end{theorem}

Note the evident fact: if $r(v)$ is of bounded characteristic in
$\Omega$ then for the poles $\{t_j\}$ the Blaschke condition
\eqref{matbl} holds.

\begin{proposition} If $\fA$ belongs to $\fae$ then
the associated Schur functions $\theta_\pm$ are of bounded
characteristic in $\Omega$ and
\begin{equation}\label{scforth}
\log|1-|\theta_\pm(v(\tau))|^2|\in L^1.
\end{equation}
\end{proposition}

\begin{proof}

We use the formula (see \eqref{Rwiththeta})
\begin{equation*}
   r_{\fA}(v):= \left\langle 0\right|\frac{\fA+v}{\fA-v} \left|0\right\rangle=
\frac{1+v\theta_+(v)\theta_-(v)}{1-v\theta_+(v)\theta_-(v)}.
\end{equation*}
Since  $\fA\in \fae$ and $r_{\fA}(v)$ is a resolvent function, its
poles satisfy the Blaschke condition.

Now we note that
\begin{equation*}
    r_{+}(v)+r_{-}(v)=\frac{1+v\theta_+(v)}
{1-v\theta_+(v)}+\frac{1+\theta_-(v)}{1-\theta_-(v)}=
\frac{2(1-v\theta_+(v)\theta_-(v))}{(1-v\theta_+(v))(1-\theta_-(v))}.
\end{equation*}
Since zeros and poles of the last function interlace we get that
poles of $r_\pm$ also satisfy the Blaschke condition. By Theorem
\ref{D} they are of bounded characteristic in $\Omega$. Hence
 $\theta_\pm$ are also in this class.

By \eqref{Wwiththeta} we get \eqref{scforth}.

\end{proof}

\begin{definition} A
function $\theta(v)$ belongs to the  class $\thte$ if it is a
function of bounded characteristic in $\Omega$ with the following
properties
\begin{equation}\label{proftheta1}
    \frac{1-\theta(v(\z))\overline{\theta(v(\z))}}{i(\z-\bar\z)}\ge 0
\end{equation}
and
\begin{equation}\label{proftheta2}
\log|1-|\theta(v(\tau))|^2|\in L^1.
\end{equation}
\end{definition}

\smallskip
Denote
\begin{equation*}
    \bbT_-=\{\t\in\bbT: \Im\,\t<0\},\quad  \bbD_-=\{\z\in\bbD:
\Im\,\z<0\}.
\end{equation*}

\begin{proposition}\label{prop5.4}
Functions  of the class $\thte$ possess the following parametric
representation
\begin{equation}\label{prthte}
    \theta(v(\z))=e^{ic}\prod_{\Lambda}\frac{\bar\lambda_k}{\lambda_k}
\frac{\lambda_k-\z}{\bar\lambda_k-\z}
\frac{1-\lambda_k\z}{1-\bar\lambda_k\z}e^{-\int_{T_-}
\left(\frac{\t+\z}{\t-\z}-\frac{\bar\t+\z}{\bar\t-\z}\right)(d\mu(\t)-\log\rho(\t)dm(\t))},
\end{equation}
where
\begin{itemize}
\item $\Lambda=\{\lambda_k\}\subset\bbD_-$ is a Blaschke sequence,
\item $\mu$ is a singular measure on the (open) set $\bbT_-$,
\item $\rho$, $0\le\rho\le 1$, is such that
\begin{equation}\label{condrho}
    \int_{T_-}
\log\{(1-\rho(\t))\rho(\t)\}dm(\t)>-\infty.
\end{equation}

\end{itemize}

\end{proposition}

\begin{proof}
First we note the symmetry
\begin{equation}\label{symmth}
    \overline{\theta(v(\bar\z))}
=\frac 1{\theta(v(\z))}
\end{equation}
and then use  the parametric representation of functions of bounded
characteristic and \eqref{proftheta2}.

In the opposite direction to prove \eqref{proftheta1} we can use
directly representation \eqref{prthte} or note that $\theta(v(\z))$
is of the Smirnov class in the domain $\bbD_-$ and then use the
maximum principle.

\end{proof}

\begin{remark}
$\theta_\pm\in\thte$ implies \eqref{matbl} and \eqref{matrsz}, but
the spectral measure $d\Sigma$ is not necessarily absolutely
continuous on $E$.

\end{remark}

\begin{example}\label{ex1}
On the other hand for every $\theta_+\in\thte$ there exists
$\theta_-\in\thte$ such that the associated to them $\fA$ belongs to
$\fae$. Put, for instance,
\begin{equation}\label{thconst}
    \theta_-(v(\z))=e^{ic_-}\frac{1-\z\bar\k}{1-\z\k},
\end{equation}
that corresponds to the constant Schur parameters (see Theorem
\ref{th1.6}). Since for every $\epsilon>0$
\begin{equation*}
    \sup_{\{\z\in \bbD_-:\Im\,\z<-\epsilon\}}|\theta_-(v(\z))|<1,
\end{equation*}
we have that both resolvent functions
\begin{equation*}
    \frac{1+v\theta_-(v)\theta_+(v)}{1-v\theta_-(v)\theta_+(v)},\quad
\frac{1+v\theta^{(-1)}_+(v)\theta^{(1)}_-(v)}{1-v\theta^{(-1)}_+(v)\theta^{(1)}_-(v)}
\end{equation*}
are uniformly bounded in the such domain. Therefore the open arc
$E\setminus\{e^{i\xi_0},e^{-i\xi_0}\}$ is free of the singular
spectrum.

Consider the end points. Existence of a mass point here means that
at least one of the following fore limits
\begin{equation*}
   \lim_{v\to e^{\pm i\xi_0}}
\frac{1+v\theta_+(v)\theta_-(v)}{1-v\theta_+(v)\theta_-(v)},\quad
 \lim_{v\to e^{\pm
i\xi_0}}\frac{1+v\theta^{(-1)}_+(v)\theta^{(1)}_-(v)}{1-v\theta^{(-1)}_+(v)\theta^{(1)}_-(v)},\quad
v\in\bbT\setminus E,
\end{equation*}
is infinite. In other words at least one of the following relations
hold
\begin{equation}\label{endpoints}
   \lim_{\z\to \pm 1}
v(\z)\theta_+(v(\z))\theta_-(v(\z))=1,\
 \lim_{\z\to\pm 1}v(\z)\theta^{(-1)}_+(v(\z))\theta^{(1)}_-(v(\z))=1,
\end{equation}
for $\z\in [-1,1]$. Due to
\begin{equation}\label{oneandzero}
   {1-v\theta^{(-1)}_+(v)\theta^{(1)}_-(v)}=\frac{\rho_{-1}^2
(1-v\theta_-(v)\theta_+(v))}{(1+a_{-1}\theta_-(v))(1+\bar
a_{-1}v\theta_+(v))}
\end{equation}
the first and the second conditions in \eqref{endpoints} are
equivalent. Thus, up to two possible exceptional values
\begin{equation*}
    e^{-ic_-}=-\frac{1\pm\bar\k}{1\pm\k}\lim_{\z\to\pm
1}\theta_+(v(\z)),
\end{equation*}
the endpoints also free of the mass of the measure $d\Sigma$.

\end{example}

 Now we prove a theorem on a representation of a Schur
function of the above class in the form similar to \eqref{sa}.

\begin{theorem}\label{reprth2}
Let $\theta(v)\in \thte$. Then there exists and unique the
representation
\begin{equation}\label{thetaL}
    \theta(v(\z))=e^{ic}\frac{L_{\bar\k}(\z)}{L_{\k}(\z)},\quad
L_{\bar\k}(\bar\k)>0,\ L_{\k}(\k)>0,
\end{equation}
such that $L_{\bar\k}(\z)$ and $L_{\k}(\z)$ are of Smirnov class in
$D$ with the mutually simple inner parts,  and the (Wronskian)
identity
\begin{equation}\label{wid0}
    \left|\begin{matrix}
\overline{L_{\k}(\t)}& \overline{ L_{\bar\k}(\t)}\\
L_{\bar\k}(\t)&L_{\k}(\t)
\end{matrix}
\right|=\frac{d\log v(\tau)}{d\log\t},\quad \t\in \bbT,
\end{equation}
holds.
\end{theorem}

\begin{proof}
By \eqref{wid0} and \eqref{thetaL} we have
\begin{equation}\label{Okap}
    |L_{\k}(\tau)|^2(1-|\theta(v(\tau))|^2)=\frac{d\log
v(\tau)}{d\log\t}.
\end{equation}
Due to \eqref{proftheta2} we can define the outer function $O_\k$
such that
\begin{equation*}
    |O_{\k}(\tau)|^2=(1-|\theta(v(\tau))|^2)^{-1}\frac{d\log
v(\tau)}{d\log\t}, \quad O_{\k}(\k)>0,
\end{equation*}
and the outer function $O_{\bar\k}$ such that
\begin{equation*}
    |O_{\bar\k}(\tau)|^2={|O_{\k}(\tau)|^2}
{|\theta(v(\tau))|^2}, \quad O_{\bar\k}(\bar\k)>0.
\end{equation*}
We represent the inner part of the function $\theta(v(\zeta))$ as
the ration of the inner {\it holomorphic} functions
\begin{equation*}
    \frac{I_{\bar\k}(\z)}{I_{\k}(\z)},\quad I_{\bar\k}(\bar\k)>0,\quad
I_{\k}(\k)>0.
\end{equation*}
Finally we put
\begin{equation*}
    L_{\bar\k}(\z):=I_{\bar\k}(\z)O_{\bar\k}(\z),\quad
L_{\k}(\z):=I_{\k}(\z)O_{\k}(\z).
\end{equation*}

Then the left-- and right-- hand sides of \eqref{thetaL} coincide up
to a unimodular constant and this defines $e^{ic}$. By \eqref{Okap}
relation \eqref{wid0} also holds.

It is evident that $L_{\bar\k}(\z)$ and $L_{\bar\k}(\z)$ as
functions of the Smirnov class are defined uniquely.

Note that due to the uniqueness and  property \eqref{symmth} we have
$\overline{L_{\bar\k}(\bar\z)}=L_{\k}(\z)$. That is \eqref{wid0} can
be written in the form similar to \eqref{wids}
\begin{equation}\label{widbar}
    \left|\begin{matrix}
 L_{\bar\k}(\bar\t)&  L_{\k}(\bar\t)\\
L_{\bar\k}(\t)&L_{\k}(\t)
\end{matrix}
\right|=\frac{d\log v(\tau)}{d\log\t}.
\end{equation}

\end{proof}

\begin{theorem}
Let $\theta\in\thte$ and let $\{a_k\}_{k=0}^\infty$ be the sequence
of its Schur parameters. Put
\begin{equation}\label{thetaLn}
    \theta^{(n)}(v(\z))=e^{ic_n}\frac{L_{\bar\k}(n,\z)}{L_{\k}(n,\z)}.
\end{equation}
Then $e^{ic_n}=e^{ic}$ and
\begin{equation}\label{recLnn}
    \begin{split}
L_{\bar\k}(n,\z)=&(e^{-ic}a_n)L_{\k}(n,\z)+\rho_n
b_\k(\z)L_{\bar\k}(n+1,\z),
\\
L_{\k}(n,\z)=&(e^{ic}\bar a_n)L_{\bar\k}(n,\z)+\rho_n
b_{\bar\k}(\z)L_{\k}(n+1,\z).
\end{split}
\end{equation}

\end{theorem}

\begin{proof} By definition
\begin{equation*}
    \theta^{(1)}=\frac{\theta-a_0}{1-\theta\bar a_0}
\frac{b_{\bar\k}}{b_{\k}}=e^{ic}\frac{L_{\bar\k}-(e^{-ic}a_0)L_{\k}}
{L_{\k}-(\overline{e^{-ic}a_0})L_{\bar\k}}\frac{b_{\bar\k}}{b_{\k}}=
e^{ic_1}\frac{ L^{(1)}_{\bar\k}}{L^{(1)}_{\k}}.
\end{equation*}
By the uniqueness of representation \eqref{thetaL} we get
\begin{equation}\label{recL}
\begin{bmatrix}
\tilde\rho_1b_\k L^{(1)}_{\bar\k}&\tilde\rho b_{\bar\k}L^{(1)}_{\k}
\end{bmatrix}=
\begin{bmatrix}
L_{\bar\k}&L_{\k}
\end{bmatrix}
\begin{bmatrix}
   1&-{e^{ic}\bar a_0}\\
-e^{-ic}a_0&   1
\end{bmatrix},
\end{equation}
with $\tilde\rho_1=\tilde\rho e^{i(c_1-c)}$. Using
\begin{equation*}
 a_0=e^{ic}\frac{L_{\bar\k}(\k)}{L_{\k}(\k)},
\end{equation*}
we have in particular
\begin{equation*}
    \tilde\rho_1
b_{\k}(\bar\k)L^{(1)}_{\bar\k}(\bar\k)=L_{\bar\k}(\bar\k)(1-|a_0|^2),\quad
\tilde\rho b_{\bar\k}(\k)L^{(1)}_{\k}(\k)=L_{\k}(\k)(1-|a_0|^2).
\end{equation*}
That is, both $\tilde\rho, \tilde\rho_{1}$ are positive and
therefore  $\tilde\rho_1=\tilde\rho$ or $e^{ic_1}=e^{ic}$.

From \eqref{recL} we have the matrix identity
\begin{equation}\label{recLmtrx}
\tilde\rho\begin{bmatrix}
    b_\k(\bar\t)L^{(1)}_{\bar\k}(\bar\t)&b_{\bar\k}(\bar\t)L^{(1)}_{\k}(\bar\t)\\
b_\k(\t)L^{(1)}_{\bar\k}(\t)&b_{\bar\k}(\t)L^{(1)}_{\k}(\t)
\end{bmatrix}=
\begin{bmatrix}
    L_{\bar\k}(\bar\t)&L_{\k}(\bar\t) \\
L_{\bar\k}(\t)&L_{\k}(\t)
\end{bmatrix}
\begin{bmatrix}
    1&-{e^{ic}\bar a_0}\\
-{e^{-ic}a_0}&   1
\end{bmatrix}.
\end{equation}
Finally using \eqref{widbar} we have $\tilde\rho^2=1-|a_0|^2$. Hence
$\tilde\rho=\rho_0$. Thus \eqref{recLnn} holds for $n=0$ and we can
iterate this procedure.
\end{proof}

\begin{lemma} For the spectral density $W$ the following
factorization holds
\begin{equation}\label{winv1}
    W^{-1}(v(\t))\frac{dm(v(\tau))}{dm(\tau)}=
\rho_{-1}^2\Phi(\t)\Phi^*(\t),
\end{equation}
where
\begin{equation}\label{Phi2}
    \Phi(\tau)=\begin{bmatrix} \frac 1{b_{\k}(\tau)}{L_{-,\k}(\tau)}&-e^{ic_+}L_{+,\bar \k}(\tau)\\
-e^{ic_-}L^{(1)}_{-,\bar\k}(\tau)& \frac
1{b_{\k}(\tau)}{L^{(-1)}_{+,\k}(\tau)}
\end{bmatrix}.
\end{equation}
\end{lemma}

\begin{proof}
Due to \eqref{thetaL}
\begin{equation}\label{tk+}
    \theta_+(v)=e^{ic_+}\frac{L_{+,\bar\k}(\z)}{L_{+,\k}(\z)},\quad
    \theta_-^{(1)}(v)=e^{ic_-}\frac{L^{(1)}_{-,\bar \k}(\z)}{L^{(1)}_{-,\k}(\z)}.
\end{equation}
Besides, due to \eqref{Okap}
\begin{equation}\label{mtk+}
  1-  |\theta_+(v(\tau))|^2=\frac
1{|L_{+,\k}(\tau)|^2}\frac{dm(v(\tau))}{dm(\tau)}
\end{equation}
and
\begin{equation}\label{mtk-}
  1-  |\theta_-^{(1)}(v(\tau))|^2=\frac
1{|L^{(1)}_{-,\k}(\tau)|^2}\frac{dm(v(\tau))}{dm(\tau)}.
\end{equation}
By definition \eqref{Wwiththeta} and \eqref{mtk-}, \eqref{mtk+}, we
have
\begin{equation*}
    W^{-1}(v(\t))\frac{dm(v(\tau))}{dm(\tau)}=
\frac 2{I+\cR(v)}
\begin{bmatrix}|L^{(1)}_{-,\k}(\tau)|^2&0\\
0&|L_{+,\k}(\tau)|^2
\end{bmatrix}
\frac 2{I+\cR^*(v(\t))}.
\end{equation*}
By definition \eqref{Rwiththeta}
\begin{equation*}
\frac 2{I+\cR(v)}=I-vA_{-1}^*
\begin{bmatrix}
\theta_-^{(1)}&0\\0&\theta_+
\end{bmatrix}.
\end{equation*}
Therefore we get \eqref{winv1} with
\begin{equation}\label{Phi}
   \rho_{-1} \Phi(\tau)=\frac{b_{\bar\k}}{b_{\k}}\begin{bmatrix} L^{(1)}_{-,\k}&0\\
0&L_{+,\k}
\end{bmatrix}(\tau)- A^*_{-1}
\begin{bmatrix} e^{ic_-}L^{(1)}_{-,\bar\k}&0\\
0&e^{ic_+}L_{+,\bar\k}
\end{bmatrix}(\tau).
\end{equation}

By \eqref{recLnn}
\begin{equation*}
\begin{split}
&\rho_{-1} b_{\bar\k}(\z) L_{+, \k}(\z)
=L^{(-1)}_{+, \k}(\z)-e^{ic_+}\bar a_{-1}L^{(-1)}_{+,\bar\k}(\z)\\
=&L^{(-1)}_{+, \k}(\z)-e^{ic_+}\bar a_{-1}(
e^{-ic_+} a_{-1}L^{(-1)}_{+, \k}(\z)+\rho_{-1}b_{\k}(\z)L_{+,\bar\k}(\z))\\
=& (\rho_{-1})^2L^{(-1)}_{+, \k}(\z) -\rho_{-1}e^{ic_+}\bar
a_{-1}b_{\k}(\z)L_{+,\bar\k}(\z),
\end{split}
\end{equation*}
that is
\begin{equation}\label{bkp}
b_{\bar\k}(\z) L_{+,\k}(\z)+e^{ic_+}\bar
a_{-1}b_{\k}(\z)L_{+,\bar\k}(\z)= \rho_{-1}L^{(-1)}_{+,\k}(\z).
\end{equation}
and similarly
\begin{equation}\label{bkm}
    b_{\bar\k}(\z) L^{(1)}_{-, \k}(\z)+e^{ic_-}\bar a^{(1)}_{-,-1}b_{\k}(\z)L^{(1)}_{-,\bar\k}(\z)=
\rho^{(1)}_{-,-1}L_{-, \k}(\z),
\end{equation}
Note that  $a_{-,-1}^{(1)}=   -\bar a_{-1}$ (generally
$a_{-,k}^{(1)}=   -\bar a_{-k-2}$).
Thus, using \eqref{bkm}, \eqref{bkp}, we get \eqref{Phi2} from
\eqref{Phi}. The lemma is proved.

\end{proof}

\begin{lemma}
Define
\begin{equation}\label{scat}
    S(\tau)=\begin{bmatrix} R_-& T_-\\
T_+& R_+
\end{bmatrix}=-\Phi^{-1}(\tau)\bar\tau\Phi(\bar\tau).
\end{equation}
Then \eqref{5.9d}, \eqref{2.10d}, \eqref{8.9d} and \eqref{normob}
hold true.
\end{lemma}
\begin{proof}
 $S(\t)$ is unitary--valued since $W(v(\bar\t))=W(v(\t))$.

Due to $\overline{L_{\pm,\bar\k}(\bar\z)}={L_{\pm,\k}(\z)}$ and
$\bar A_n=A_n^*=A_n^{-1}$ we get directly from \eqref{Phi}
\begin{equation}\label{phisym}
    \overline{\Phi(\bar \tau)}=-v A_{-1}\Phi(\tau)
\begin{bmatrix}
e^{-ic_-}&0\\0&e^{-ic_+}\end{bmatrix}.
\end{equation}
And, therefore, the following symmetry property of $S$
\begin{equation}\label{Ssym}
\begin{split}
    \overline{S(\bar \tau)}=-&
\begin{bmatrix}
e^{ic_-}&0\\0&e^{ic_+}\end{bmatrix}\Phi(\tau)^{-1}\bar\tau\Phi(\bar\tau)\begin{bmatrix}
e^{-ic_-}&0\\0&e^{-ic_+}\end{bmatrix}\\=&
\begin{bmatrix}
e^{ic_-}&0\\0&e^{ic_+}\end{bmatrix}S(\t)\begin{bmatrix}
e^{-ic_-}&0\\0&e^{-ic_+}\end{bmatrix}.
\end{split}
\end{equation}
is proved.

Let us show that $T_+(\t)=\overline{T_-(\bar\t)}$. We have
\begin{equation}\label{scatelem}
\begin{bmatrix} R_-& T_-\\
T_+& R_+
\end{bmatrix}=\frac{-\bar\tau}{\Delta}
\begin{bmatrix} \frac
1{b_{\k}}{L^{(-1)}_{+,\k}}&e^{ic_+}L_{+,\bar
\k}\\
e^{ic_-}L^{(1)}_{-,\bar\k} & \frac 1{b_{\k}}{L_{-,\k}}
\end{bmatrix}(\tau)
\begin{bmatrix} \frac 1{b_{\k}}{L_{-,\k}}&-e^{ic_+}L_{+,\bar \k}\\
-e^{ic_-}L^{(1)}_{-,\bar\k}& \frac 1{b_{\k}}{L^{(-1)}_{+,\k}}
\end{bmatrix}(\bar\tau),
\end{equation}
where $\Delta=\det\Phi$. Therefore
\begin{equation}\label{s12}
\begin{split}
    T_+=&-e^{ic_-}\bar\tau\frac{
\left\vert\begin{matrix}\frac 1{b_{\k}(\bar\t)}{L_{-,\k}}(\bar\t)&L^{(1)}_{-,\bar\k}(\bar\t)\\
\frac 1{b_{\k}(\t)}{L_{-,\k}}(\t) &L^{(1)}_{-,\bar\k}(\t)
\end{matrix}\right\vert}{\Delta}\\
=&-e^{ic_-}\bar\tau\frac{ \left\vert\begin{matrix}
v^{-1}L^{(1)}_{-,\k}(\bar\tau)-
e^{ic_-} a_{-1}L^{(1)}_{-,\bar \k}(\bar\tau)&L^{(1)}_{-,\bar\k}(\bar\tau) \\
 v^{-1}L^{(1)}_{-,\k}(\tau)-
e^{ic_-} a_{-1}L^{(1)}_{-,\bar\k}(\tau)&L^{(1)}_{-,\bar\k}(\tau)
\end{matrix}\right\vert}{\rho_{-1}\Delta}
=-e^{ic_-}\frac{(v^{-1})'}{\rho_{-1}\Delta}.
\end{split}
\end{equation}
Similarly $T_-=-e^{ic_+}\frac{(v^{-1})'}{\rho_{-1}\Delta}=
e^{i(c_+-c_-)}T_+$. Due to \eqref{Ssym}
$\overline{T_+(\bar\t)}=e^{i(c_+-c_-)}T_+(\t)$, therefore the
symmetry $S^*(\bar\tau)=S(\t)$ is completely proved.

Note also that \eqref{s12} implies the following normalization
\begin{equation}\label{normtp}
    T_+(\k)=e^{ic_-}\frac{b_{\bar\k}(\k)b'_{\k}(\k)}{\rho_{-1}L^{(-1)}_{+,\k}(\k)L_{-,\k}(\k)}=
e^{ic_-}\frac{b'_{\k}(\k)}{L^{(-1)}_{+,\k}(\k)L^{(1)}_{-,\k}(\k)}.
\end{equation}
That is, $T_+(\k)=-ie^{ic_-}|T_+(\k)|$.

Finally we have to prove that $T_\pm$ is a ratio of an outer
function and a Blaschke product. In other words, by \eqref{s12}, we
need to show that the inner part of the Smirnov class function
\begin{equation}\label{determinant}
\begin{split}
   \rho_{-1} b^2_{\k}\Delta=&
\det\begin{bmatrix}L_{-,\k}
&-e^{ic_+}{b_{\k}}L_{+,\bar\k}\\
-e^{ic_-}\rho_{-1}{b_{\k}}L^{(1)}_{-,\bar\k}&
\rho_{-1}L^{(-1)}_{+,\k}\end{bmatrix}\\
=&\left|
\begin{matrix}L_{-,\k}
&-e^{ic_+}{b_{\k}}L_{+,\bar\k}\\
-e^{ic_-}(L_{-,\bar\k}+\bar a_{-1}e^{-ic_-}L_{-,\k})&
{b_{\bar\k}}L_{+,\k}+\bar
a_{-1}e^{ic_+}{b_{\k}}L_{+,\bar\k}\end{matrix}\right|\\
=&\left|
\begin{matrix}L_{-,\k}
&-e^{ic_+}{b_{\k}}L_{+,\bar\k}\\
-e^{ic_-}L_{-,\bar\k}& {b_{\bar\k}}L_{+,\k}\end{matrix}\right|\\
=&{b_{\bar\k}}L_{-,\k}L_{+,\k}-e^{i(c_+
+c_-)}{b_{\k}}L_{+,\bar\k}L_{-,\bar\k}
\end{split}
\end{equation}
is a Blaschke product (actually related to the spectrum of the
associated CMV matrix).

Since
\begin{equation*}
\frac{{b_{\bar\k}}L_{-,\k}L_{+,\k}+e^{i(c_+
+c_-)}{b_{\k}}L_{+,\bar\k}L_{-,\bar\k}}
{{b_{\bar\k}}L_{-,\k}L_{+,\k}-e^{i(c_+
+c_-)}{b_{\k}}L_{+,\bar\k}L_{-,\bar\k}}
=\frac{1+v\theta_+\theta_-}{1-v\theta_+\theta_-}
\end{equation*}
by Theorem \ref{D} the inner part of this fraction is a ration of
two Blaschke products. Thus, any other inner divisor of the inner
part of $\rho_{-1}b^2_{\k}\Delta$ should simultaneously divide the
inner part of the numerator ${b_{\bar\k}}L_{-,\k}L_{+,\k}+e^{i(c_+
+c_-)}{b_{\k}}L_{+,\bar\k}L_{-,\bar\k}$. That is,
${b_{\bar\k}}L_{-,\k}L_{+,\k}$ and $e^{i(c_+
+c_-)}{b_{\k}}L_{+,\bar\k}L_{-,\bar\k}$ possess a nontrivial common
inner factor in  this case. But they are coprime since the inner
part of the first function is supported in the upper half plane and
of the second one in the lower part, see Proposition \ref{prop5.4}.

\end{proof}

\begin{lemma} For every $\z_k\in\cZ$ the following two vectors are
collinear
\begin{equation}\label{collin}
  \begin{bmatrix}
e^{ic_-}L^{(1)}_{-,\bar\k}&\frac 1{b_{\k}}{L_{-,\k}}
\end{bmatrix}(\z_k)
=-\left(\frac 1 {T_-}\right)'(\z_k)\nu_+(\z_k)
  \begin{bmatrix}\frac
1{b_{\k}}{L^{(-1)}_{+,\k}} &e^{ic_+}L_{+,\bar \k}
\end{bmatrix}(\z_k).
\end{equation}
Moreover $\nu_+(\z_k)>0$.

\end{lemma}

\begin{proof}By definition \eqref{ir} and \eqref{Rwiththeta} we have
$$
t_k\Sigma(t_k)= \left\{{(t_k-v)}\left(I-vA^*_{-1}\begin{bmatrix}
\theta_-^{(1)}(v)&0\\0&\theta_+(v)
\end{bmatrix}\right)^{-1}\right\}_{v=t_k}.
$$
Since
$$
\rho_{-1}v\Phi(\t)=
 \left(I-vA^*_{-1}\begin{bmatrix}
\theta_-^{(1)}(v)&0\\0&\theta_+(v)
\end{bmatrix}\right)
\begin{bmatrix} L^{(1)}_{-,\k}&0\\
0&L_{+,\k}
\end{bmatrix}(\tau),
$$
we get
\begin{equation}
\begin{split}
t_k\Sigma(t_k)&=\begin{bmatrix} L^{(1)}_{-,\k}&0\\
0&L_{+,\k}
\end{bmatrix}(\z_k)\left
\{\frac {(t_k-v)}{\rho_{-1}v}\Phi^{-1}(\t)\right\}_{\t=\z_k}
\\
&=
\begin{bmatrix} L^{(1)}_{-,\k}&0\\
0&L_{+,\k}
\end{bmatrix}(\z_k)
\begin{bmatrix}\frac
1{b_{\k}}{L^{(-1)}_{+,\k}}
&e^{ic_+}L_{+,\bar \k}\\
e^{ic_-}L^{(1)}_{-,\bar\k}&\frac 1{b_{\k}}{L_{-,\k}}
\end{bmatrix}(\z_k)
\left\{\frac {(t_k-v)}{\rho_{-1}v\Delta}\right\}_{\t=\z_k}
\\
&=
\begin{bmatrix} L^{(1)}_{-,\k}&0\\
0&L_{+,\k}
\end{bmatrix}(\z_k)
\begin{bmatrix}\frac
1{b_{\k}}{L^{(-1)}_{+,\k}}
&e^{ic_+}L_{+,\bar \k}\\
e^{ic_-}L^{(1)}_{-,\bar\k}&\frac 1{b_{\k}}{L_{-,\k}}
\end{bmatrix}(\z_k)
\frac {-v'(\z_k)}{\rho_{-1}t_k\Delta'(\z_k)},
\end{split}
\end{equation}
or, using $T_-=-e^{ic_+}\frac{(v^{-1})'}{\rho_{-1}\Delta}$,
\begin{equation}\label{sigmaL}
    \Sigma(t_k)=
\begin{bmatrix} L^{(1)}_{-,\k}&0\\
0&L_{+,\k}
\end{bmatrix}(\z_k)
\begin{bmatrix}\frac
1{b_{\k}}{L^{(-1)}_{+,\k}}
&e^{ic_+}L_{+,\bar \k}\\
e^{ic_-}L^{(1)}_{-,\bar\k}&\frac 1{b_{\k}}{L_{-,\k}}
\end{bmatrix}(\z_k)
\left\{-\left(\frac {e^{ic_+}} {T_-}\right)'(\z_k)\right\}^{-1}.
\end{equation}
From this formula we conclude the vector in the RHS \eqref{collin}
does not vanish. Otherwise, by
$L_{+,\k}(\z_k)=\overline{L_{+,\bar\k}(\z_k)}$, we have
$\Sigma(t_k)=0$, which is impossible. On the other hand rank of the
second matrix in \eqref{sigmaL} is one, therefore \eqref{collin} is
proved.

Now, making of use \eqref{collin} and the symmetry of $T_-$, we get
from \eqref{sigmaL}
\begin{equation}\label{sigmaLsym}
    \Sigma(t_k)=
\begin{bmatrix}\overline{\frac
1{b_{\k}(\z_k)}{L^{(-1)}_{+,\k}(\z_k)}}
\\e^{-ic_+}\overline{L_{+,\bar \k}(\z_k)}
\end{bmatrix}
\begin{bmatrix}\frac
1{b_{\k}(\z_k)}{L^{(-1)}_{+,\k}(\z_k)} &e^{ic_+}L_{+,\bar \k}(\z_k)
\end{bmatrix}
\nu_+(\z_k),
\end{equation}
here $\Sigma(t_k)\ge 0$ implies $\nu_+(\z_k)>0$.

\end{proof}

\begin{remark}
Similarly
\begin{equation}\label{collinm}
 -\left(\frac 1 {T_+}\right)'(\z_k)\nu_-(\z_k) \begin{bmatrix}
e^{ic_-}L^{(1)}_{-,\bar\k}&\frac 1{b_{\k}}{L_{-,\k}}
\end{bmatrix}(\z_k)
=
  \begin{bmatrix}\frac
1{b_{\k}}{L^{(-1)}_{+,\k}} &e^{ic_+}L_{+,\bar \k}
\end{bmatrix}(\z_k).
\end{equation}
Therefore \eqref{3.10d} holds for $\nu_\pm$ defined by
\eqref{collin} and \eqref{collinm}.

\end{remark}

\section{From the spectral representation to the scattering representation }

In this section an essential part of Theorem \ref{thszego} will be
proved.

\begin{theorem}\label{th61}
Let $\fA\in\fae$. Define $S$ by \eqref{scat} and $\nu_\pm$ by
\eqref{collin} and \eqref{collinm}. Then
\begin{equation}\label{tls04bis}
e^\pm(n,\z)=\begin{cases} b_\k^{m}(\z)
b_{\bar\k}^{m}(\z)L_{\pm,\bar\k}(n,\z)e^{ic_\pm} ,
&n=2m\\
b_\k^{m}(\z) b_{\bar\k}^{m+1}(\z)L_{\pm,\k}(n,\z) ,&n=2m+1
\end{cases}
\end{equation}
is an orthonormal basis in $L^2_{\alpha_\pm}$,
$\alpha_\pm=\{R_\pm,\nu_\pm\}$.
\end{theorem}
The proof is based on the following lemma.

\begin{lemma}
For $f^+\in L^2_{\alpha_+}$
\begin{equation}\label{mar31}
\begin{bmatrix}
    \left\langle
\frac{v(\t)+w}{v(\t)-w}f^+,e^+(-1,\t)\right\rangle_{\alpha_+}\\
\left\langle
\frac{v(\t)+w}{v(\t)-w}f^+,e^+(0,\t)\right\rangle_{\alpha_+}
\end{bmatrix}=
\int\frac{t+w}{t-w}d\Sigma(t)\tilde f(t),
\end{equation}
where
\begin{equation}\label{mar32}
    \tilde f(t):=\frac 1{\Delta(\t)}\Phi
\begin{bmatrix}
          {f^+}\\  f^-
               \end{bmatrix}
(\t),\quad t=v(\t), \ \t\in \bbT_-,
\end{equation}
\begin{equation}\label{mar33}
    \tilde f(t_k):=\begin{bmatrix}\overline{e^+(-1,\z_k)}\\
\overline{e^+(0,\z_k)}\end{bmatrix}\frac{f^+(\z_k)}
{|e^+(-1,\z_k)|^2+ |e^+(0,\z_k)|^2}, \quad t_k=v(\z_k), \ \z_k\in
\cZ.
\end{equation}

\end{lemma}

\begin{proof}
Note that in this notations (see \eqref{sigmaLsym})
\begin{equation}\label{sgmae}
    \Sigma(t_k)=
\begin{bmatrix}\overline{e^+(-1,\z_k)}\\
\overline{e^+(0,\z_k)}
\end{bmatrix}
\begin{bmatrix}e^+(-1,\z_k)&e^+(0,\z_k)
\end{bmatrix}\nu_+(\z_k),
\end{equation}
and (see \eqref{scat})
\begin{equation}\label{scate}
    \begin{bmatrix} e^-(-1,\t)&-e^+(0,\t)\\
-e^-(0,\t) &e^+(-1,\t)
\end{bmatrix}
\begin{bmatrix} R_-& T_-\\
T_+& R_+
\end{bmatrix}(\t)=-\bar\t
\begin{bmatrix} e^-(-1,\bar\t)&-e^+(0,\bar\t)\\
-e^-(0,\bar\t) &e^+(-1,\bar\t)
\end{bmatrix}.
\end{equation}
Therefore, by definition of the scalar product in $L^2_{\alpha_+}$,
we have
\begin{equation*}
\begin{split}
&\begin{bmatrix}\left\langle
\frac{v(\t)+w}{v(\t)-w}f^+,e^+(-1,\t)\right\rangle\\
\left\langle \frac{v(\t)+w}{v(\t)-w}f^+,e^+(0,\t)\right\rangle
\end{bmatrix}=
 \sum_{\z_k\in\cZ}
\begin{bmatrix}\overline{e^+(-1,\z_k)}f^+(\z_k)\\
\overline{e^+(0,\z_k)}f^+(\z_k)
\end{bmatrix}
\frac{v(\z_k)+w}{v(\z_k)-w}\nu_+(\z_k)\\
 +&
\int_{\bbT_-}\left\{\begin{bmatrix} T_+(\t)e^+(-1,\t)&T_+(\t)e^+(0,\t)\\
T_-(\t) e^-(0,\t)&T_-(\t) e^-(-1,\t)
\end{bmatrix}\right\}^*
\begin{bmatrix}
          {T_+f^+}\\  T_-f^-
               \end{bmatrix}
(\t)\frac{v(\t)+w}{v(\t)-w}dm(\t).
\end{split}
\end{equation*}
Using  \eqref{sgmae}, definition \eqref{mar33} and
\begin{equation}\label{phie}
    \Phi^{-1}(\t)=\frac 1{\Delta}\begin{bmatrix}e^+(-1,\t)&e^+(0,\t)\\
e^-(0,\t)&e^-(-1,\t)
\end{bmatrix},
\end{equation}
we get
\begin{equation*}
\begin{split}
\begin{bmatrix}
\left\langle
\frac{v(\t)+w}{v(\t)-w}f^+,e^+(-1,\t)\right\rangle\\
\left\langle \frac{v(\t)+w}{v(\t)-w}f^+,e^+(0,\t)\right\rangle
\end{bmatrix}
=& \sum_{t_k\in X}\frac{t_k+w}{t_k-w}\Sigma(t_k)\tilde f(t_k)
\\
 +&
\int_{\bbT_-}\overline {\Delta(\t)}\left\{\Phi^{-1}(\t) \right\}^*
\begin{bmatrix}
          |T_+|^2{f^+}\\  |T_-|^2f^-
               \end{bmatrix}
(\t)\frac{v(\t)+w}{v(\t)-w}dm(\t)\\
= \sum_{t_k\in X}\frac{t_k+w}{t_k-w}\Sigma(t_k) \tilde f(t_k)
 +&
\int_{E}\frac{t+w}{t-w}W(t)\left\{\frac{1}{\Delta}\Phi
\begin{bmatrix}
          {f^+}\\  f^-
               \end{bmatrix}\right\}
(\t)dm(t),
\end{split}
\end{equation*}
since $ W=(\Phi^{-1})^*\Phi^{-1}\frac{|v'|}{\rho_{-1}^2}$ and
$|T_-|^2=|T_+|^2=\frac{|v'|^2}{\rho_{-1}^2|\Delta|^2}$.

\end{proof}

\begin{proof}[Proof of Theorem \ref{th61}]
It was shown that $e^+(-1,\t)$, $e^+(0,\t)$ form a cyclic subspace
for the multiplication operator by $v(\t)$ in $L^2_{\alpha_+}$,
moreover, the resolvent matrix function
\begin{equation*}
    (\cE^+)^*\frac{v(\t)+w}{v(\t)-w}\cE^+,\quad \cE^+
\begin{bmatrix}
c_{-1}\\c_0
\end{bmatrix}:=e^+(-1,\t)c_{-1}+ e^+(0,\t)c_{0},
\end{equation*}
coincides with $\cR(w)$ \eqref{rf}. Therefore the operator
$\cF^+:l^2(\bbZ)\to L^2_{\alpha_+}$, defined by
\begin{equation*}
    \cF^+(\fA-w)^{-1}|n\rangle=(v(\t)-w)^{-1}e^+(n,\t), \quad
n=-1,0,
\end{equation*}
is unitary.

Recurrences \eqref{recLnn} implies \eqref{reprapeven},
\eqref{reprapodd}, and therefore, \eqref{evp}. Thus
\begin{equation*}
    \cF^+|n\rangle=e^+(n,\t), \quad
n\in\bbZ,
\end{equation*}
and the theorem is proved.
\end{proof}

\begin{proof}[Proof of Proposition \ref{prop120}]
Note that $e^+(n,\t)$'s are in the Smirnov class for $n\in \bbZ_+$.
Therefore $(BT_+)(\t)e^+(n,\t)\in L^2$ implies
$(BT_+)(\t)e^+(n,\t)\in H^2$. Thus
\begin{equation*}
    \cF^+(\bbZ_+)\subset \hat{H}^2_{\alpha_+}.
\end{equation*}

In the same way $\cF^-(\bbZ_-)\subset \hat{H}^2_{\alpha_-}$.
Therefore, due to the duality Theorem \ref{t1.3},
\begin{equation*}
   \check{H}^2_{\alpha_+}\subset  \cF^+(\bbZ_+).
\end{equation*}

\end{proof}

\begin{remark}
Let us note the following fact
\begin{equation}\label{limh}
    \lim_{n\to-\infty}\cF^+(\bbZ_{+,n})=L^2_{\alpha_+}, \quad
\lim_{n\to\infty}\cF^+(\bbZ_{+,n})=\{0\},
\end{equation}
where $\bbZ_{+,n}:=\{m\in\bbZ, m\ge n\}$.

Also, in the standard way,
\begin{equation}\label{lrepk}
    l^{n,+}(\z,\z_0):=\sum_{m=n}^\infty
e^+(m,\z)\overline{e^+(m,\z_0)},\quad \z,\z_0 \in\bbD,
\end{equation}
is the reproducing kernel in $\cF^+(\bbZ_{+,n})$.
 In particular,
\begin{equation*}
L_{+,\bar\k}(\z)=\frac{l^{+}(\z,\bar\k)}{\sqrt{l^{+}(\bar\k,\bar\k)}},
 \quad
L_{+,\k}(\z)=\frac{l^{+}(\z,\k)}{\sqrt{l^{+}(\k,\k)}},
\end{equation*}
where $l^{+}(\z,\z_0):=l^{0,+}(\z,\z_0)$.
\end{remark}

\begin{proof}[Proof of \eqref{defnup}] Let
\begin{equation*}
    \delta_k(\t)=\begin{cases}\frac 1{\nu_+(\z_k)},&\t=\z_k,\\
0,&\t\in(\bbT\bigcup\cZ)\setminus\{\z_k\}.
\end{cases}
\end{equation*}
Then
\begin{equation*}
    \langle f^+(\t),\delta_k(\t)\rangle_{\alpha_+}=f^+(\z_k)
\end{equation*}
for every $f^+\in\cF^+(\bbZ_{+,n})$. Therefore the projection  of
$\delta_k(\t)$ onto $\cF^+(\bbZ_{+,n})$ is the reproducing kernel
$l^{n,+}(\t,\z_k)$. Since by \eqref{limh}
\begin{equation*}
    \|\delta_k(\t)\|=\lim_{n\to-\infty}\|l^{n,+}(\t,\z_k)\|,
\end{equation*}
we get by \eqref{lrepk}
\begin{equation*}
    \frac{1}{\nu_+(\z_k)}=\lim_{n\to-\infty}\sum_{m=n}^\infty
|e^+(m,\z_k)|^2=\sum_{m=-\infty}^\infty |e^+(m,\z_k)|^2.
\end{equation*}

\end{proof}

Thus, to complete the proof of Theorem \ref{thszego}, we have to
show asymptotics \eqref{asppm}.

\section{Asymptotics}

In this section we prove the main claim of Theorem \ref{thszego}.
Recall briefly notations. With $\fA\in\fae$ we associate the Schur
functions $\theta_+$, $\theta_-^{(1)}$. They belong to $\thte$ and,
therefore, possess the special representation \eqref{thetaL}. We put
\begin{equation}\label{e01}
\begin{split}
    e^{+}(0,\t)=e^{ic_+}L_{+,\bar\k}(\t),\quad
\rho_{-1}e^{+}(-1,\t)= \frac 1 {v(\t)} L_{+,\k}(\t)+e^{ic_+}\bar
a_{-1}L_{+,\bar\k}(\t),\\
e^{-}(0,\t)=e^{ic_-}L^{(1)}_{-,\bar\k}(\t),\quad
\rho_{-1}e^{-}(-1,\t)= \frac 1 {v(\t)} L^{(1)}_{-,\k}(\t)-e^{ic_-}
a_{-1}L^{(1)}_{-,\bar\k}(\t).
\end{split}
\end{equation}
The scattering matrix $S$ is defined by \eqref{scate} and the
measures $\nu_{\pm}$ on $\cZ$ are defined by
\begin{equation*}
    \nu_{\pm}(\z_k)(|e^{\pm}(-1,\z_k)|^2+|e^{\pm}(0,\z_k)|^2)=\tr\,
\Sigma(t_k).
\end{equation*}
Our goal is to prove asymptotics \eqref{asppm}, \eqref{asmpm} for
the systems defined by recurrence relations \eqref{evp}, \eqref{evm}
with the initial data $e^{\pm}(n,\t)$, $n=-1,0$.

This is a standard fact that such asymptotics can be obtained from a
convergence of a certain system of analytic functions just in a one
fixed point of their domain. More specifically, our first step is a
reduction to the convergence of the reproducing kernels
$L_{\pm,\z_0}(n,\z_0)$ to the standard one $K_{\z_0}(\z_0)$ in a
fixed point of the unite disk.

\begin{lemma} Let $\chi_{\bbT}(\t)$, $\t\in\bbT\cup\cZ$, be the characteristic function of the set $\bbT$.
Then \eqref{asppm}, \eqref{asmpm} can be deduced from
\begin{equation}\label{asl}
    \lim_{n\to\infty}
\|\chi_{\bbT}\{L_{\pm,\z_0}(n,\t)-K_{\z_0}(\t)\}\|_{{\alpha^{(n)}_\pm}}=0,
\quad\z_0\in\bbD.
\end{equation}
\end{lemma}

\begin{proof}
Using definition \eqref{scate} and the recurrence relations for
$e^\pm(n,\t)$ we have
\begin{equation*}
    \bar\t e^\pm(n,\bar\t)+R_\pm(\t)e^\pm(n,\t)=
T_{\mp}(\t)e^\mp(-n-1,\t).
\end{equation*}
Therefore conditions \eqref{asppm}, \eqref{asmpm} are equivalent to
\begin{equation*}
    T_{\pm}(\t)e^\pm(n,\t)=T_\pm(\t)\fe_{n,c_\pm}(\t)+o(1),
\end{equation*}
\begin{equation*}
    \bar\t e^\pm(n,\bar\t)+R_\pm(\t)e^\pm(n,\t)=
\bar\t\fe_{n,c_\pm}(\bar\t)+R_\pm(\t)\fe_{n,c_\pm}(\t)+o(1)
\end{equation*}
in $L^2$ as $n\to\infty$. That is,
\begin{equation}\label{asvec}
    \lim_{n\to\infty}
\|\chi_{\bbT}\{e^\pm(n,\t)-\fe_{n,c_\pm}(\t)\}\|_{L^2_{\alpha_+}}=0.
\end{equation}
Then we use \eqref{tls04bis}, \eqref{ts4bis} and definition
\eqref{shift} to rewrite \eqref{asvec} into the form
\begin{equation*}
\begin{split}
    \lim_{n\to\infty}
\|\chi_{\bbT}\{L_{\pm,\bar\k}(n,\t)-K_{\bar\k}(\t)\}\|_{{\alpha^{(n)}_\pm}}&=0,\\
\lim_{n\to\infty}
\|\chi_{\bbT}\{L_{\pm,\k}(n,\t)-K_{\k}(\t)\}\|_{{\alpha^{(n)}_\pm}}&=0.
\end{split}
\end{equation*}
\end{proof}

\begin{lemma}
Assume that
\begin{equation}\label{limll}
    \lim_{n\to\infty} L_{\pm,\z_0}(n,\z_0)= K_{\z_0}(\z_0),
\quad\z_0\in\bbD.
\end{equation}
Then
\begin{equation}\label{asl2}
    \lim_{n\to\infty}
\|L_{\pm,\z_0}(n,\t)-\chi_{\bbT}K_{\z_0}(\t)\|_{{\alpha^{(n)}_\pm}}=0.
\end{equation}
\end{lemma}

\begin{proof}
For $\epsilon>0$ chose $N$ such that $\Re\,B_N(\z_0)\ge 1-\epsilon$.
Note that $B_N K_{\z_0}\in \check{H}^2_{\alpha^{(n)}_\pm}$ and
consider
\begin{equation*}
\|L_{\pm,\z_0}(n,\t)-(B_N K_{\z_0})(\t)\|_{\alpha^{(n)}_\pm}^2,
\quad \|(B_N
K_{\z_0})(\t)-\chi_{\bbT}K_{\z_0}(\t)\|_{\alpha^{(n)}_\pm}^2.
\end{equation*}
For the first term we have
\begin{equation*}
\begin{split}
\|L_{\pm,\z_0}(n,\t)-(B_N K_{\z_0})(\t)\|_{\alpha^{(n)}_\pm}^2 =&
2-2\Re\frac{(B_N K_{\z_0})(\z_0)}{L_{\pm,\z_0}(n,\z_0)}
\\
-\Re&\langle P_- b^n_{\k}b^n_{\bar\k}R_\pm B_NK_{\z_0},\bar\t(B_N
K_{\z_0})(\bar\t)\rangle_{L^2}.
\end{split}
\end{equation*}
Note that for any two functions $f,g\in L^2$ the following limit
exists
\begin{equation}\label{pminuspr}
  \lim_{n\to\infty}\langle P_-
b^n_{\k}b^n_{\bar\k}f,g\rangle_{L^2}=0.
\end{equation}
Therefore, if \eqref{limll} is satisfied then
\begin{equation}\label{pervyj}
\limsup_{n\to\infty}\|L_{\pm,\z_0}(n,\t)-(B_N
K_{\z_0})(\t)\|_{\alpha^{(n)}_\pm}^2 \le 2\epsilon.
\end{equation}

Similarly,
\begin{equation*}
\begin{split}
 \|(B_N
K_{\z_0})(\t)-\chi_{\bbT}K_{\z_0}(\t)\|_{\alpha^{(n)}_\pm}^2=&\sum_{\cZ}
|B_N K_{\z_0}|^2(\z_k)\nu_{\pm}(\z_k)|b^n_{\k}(\z_k)|^2\\
+2-2\Re\, B_N(\z_0)+\Re&\langle P_-
b^n_{\k}b^n_{\bar\k}R_\pm(B_N-1)K_{\z_0},\bar\t((B_N-1)K_{\z_0})(\bar\t)\rangle_{L^2}.
\end{split}
\end{equation*}
Note that the sum over $\cZ$ here contains just a fixed finite
number of nonvanishing terms, and therefore it goes to zero as
$n\to\infty$. Thus, taking also into account \eqref{pminuspr}, we
get
\begin{equation}\label{vtoroy}
    \limsup_{n\to\infty}\|(B_N
K_{\z_0})(\t)-\chi_{\bbT}K_{\z_0}(\t)\|_{\alpha^{(n)}_\pm}^2\le
2\epsilon.
\end{equation}
Combining \eqref{pervyj} and \eqref{vtoroy} we have
\begin{equation*}
    \limsup_{n\to\infty}
\|L_{\pm,\z_0}(n,\t)-\chi_{\bbT}K_{\z_0}(\t)\|^2_{{\alpha^{(n)}_\pm}}\le
8\epsilon.
\end{equation*}
Since $\epsilon>0$  is arbitrary
the lemma is proved.
\end{proof}

\begin{remark} Note that, in addition to \eqref{asl}, \eqref{asl2}
contains
\begin{equation*}
    \lim_{n\to\infty}\sum_{\cZ}|L_{\pm,\z_0}(n,\z_k)|^2
|b_\k(\z_k)|^{2n}\nu_{\pm}(\z_k)=0.
\end{equation*}
\end{remark}

In the proof of \eqref{limll} we follow the line  that was suggested
in \cite{pyu} and then improved in \cite{VYu} and \cite{KPVYu}.
Actually, the general idea is very simple. There are two natural
steps in approximation of the given spectral data by ``regular"
ones. First, to substitute the given measure $\nu_+$ by a finitely
supported $\nu_{N,+}$. Second, to substitute $R_+$ by $q R_+$ with
$0<q<1$. Then the corresponding data produce the Hardy space which
is topologically equivalent to the standard  $H^2$. In particular
$$
\check K_{\alpha_{N,q,+}}(\z_0,\z_0)=\hat
K_{\alpha_{N,q,+}}(\z_0,\z_0).
$$

\begin{lemma} Let $\cZ$ contain a finite number of points and
$\|R_+\|_{L^{\infty}}<1$. Then the limit  \eqref{limll} exists.
\end{lemma}
\noindent
 Basically, it follows from \eqref{pminuspr} and
$|b_\k(\z_k)|^n\to 0$. It is a fairly easy task and we omit a proof
here.

Further, due to
\begin{equation*}
    \check H^2_{\alpha_{q,+}}\subset\check H^2_{\alpha_+}
\subset\cF^+(\bbZ_+)\subset\hat H^2_{\alpha_+}\subset \hat
H^2_{\alpha_{N,+}}
\end{equation*}
 we have the evident
estimations
$$
\check K_{\alpha_{q,+}}(\z_0,\z_0)\le
 L_{\alpha_+}(\z_0,\z_0)\le
 \hat K_{\alpha_{N,+}}(\z_0,\z_0).
$$
And the key point is that, due to the  duality principle,
\eqref{2.1af1} holds. It allow us to use the left or right side
estimation whenever it is convenient for us.

\begin{theorem} Let $\fA\in \fae$. For $\z_0\in\bbD$ the limit  \eqref{limll}
exists, and therefore \eqref{asppm}, \eqref{asmpm} hold true.
\end{theorem}

\begin{proof} Recall Lemma \ref{l1.4} on the relation between $\pm$ mappings and
the notations $|T_\pm(\z)|=\left|\frac{O(\z)}{B(\z)}\right|$,
where $B$ is a Blaschke product and $O$ is an outer function
\eqref{2.10d}.

We have
\begin{equation}\label{is2.1}
\begin{split}
   L_{\alpha_\pm^{(n)}}(\z_0,\z_0)&\le \hat K_{\alpha^{(n)}_\pm}(\z_0,\z_0) \le
\hat K_{\alpha^{(n)}_{N,\pm}}(\z_0,\z_0)=\frac{1}{|T_{N,\pm}(\z_0)|}
    \frac{K^2(\z_0,\z_0)}{\check K_{\alpha^{(-n-1)}_{N,\mp}}(\z_0,\z_0)}\\
    &\le
\frac{1}{|T_{N,\pm}(\z_0)|}
    \frac{K^2(\z_0,\z_0)}{\check K_{\alpha^{(-n-1)}_{N,q,\mp}}(\z_0,\z_0)}
    =
    \frac{|O_q(\z_0)|}{|O(\z_0)|}
    \hat K_{\alpha^{(n)}_{N,q,\pm}}(\z_0,\z_0).
    \end{split}
\end{equation}
And from the other side
\begin{equation}\label{is2.2}
\begin{split}
     L_{\alpha_\pm^{(n)}}(\z_0,\z_0)&\ge
\check K_{\alpha^{(n)}_\pm}(\z_0,\z_0)  \ge \check
K_{\alpha^{(n)}_{q,\pm}}(\z_0,\z_0)=\frac{1}{|T_{q,\pm}(\z_0)|}
    \frac{K^2(\z_0,\z_0)}{\hat K_{\alpha^{(-n-1)}_{q,\mp}}(\z_0,\z_0)}\\
    &\ge
\frac{1}{|T_{q,\pm}(\z_0)|}
    \frac{K^2(\z_0,\z_0)}{\hat K_{\alpha^{(-n-1)}_{q,N,\mp}}(\z_0,\z_0)}
    ={|B_{N}(\z_0)|}
    \check K_{\alpha^{(n)}_{q,N,\pm}}(\z_0,\z_0).
    \end{split}
\end{equation}
Passing to the limit in \eqref{is2.1} and \eqref{is2.2} we get
\begin{equation}\label{fihlim}
\begin{split}
    {|B_{N}(\z_0)|}K(\z_0,\z_0)&\le\liminf_{n\to\infty}
L_{\alpha_\pm^{(n)}}(\z_0,\z_0)\\
&\le\limsup_{n\to\infty} L_{\alpha_\pm^{(n)}}(\z_0,\z_0)\le
\frac{|O_q(\z_0)|}{|O(\z_0)|}K(\z_0,\z_0).
\end{split}
\end{equation}
Since
$$
\lim_{N\to\infty}|B_{N}(\z_0)|= 1\quad \text{and} \quad \lim_{q\to
1}|O_q(\z_0)|=|O(\z_0)|,
$$
 \eqref{fihlim} implies
\eqref{limll} and thus asymptotics \eqref{asppm}, \eqref{asmpm} are
proved.

\end{proof}



\section{ Hilbert transform}\label{shtr}

Recall definition \eqref{trop} of the transformation operator. In
terms of the decomposition \eqref{transferpm} the operator
$\cM_-:l^2(\bbZ_{-})\to l^2(\bbZ_{-})$  is defined by
\begin{equation}\label{tropmns}
    \cM_-=\iota^*\begin{bmatrix}M^-_{0,0}&0&0&\dots\\
M^-_{1,0}&M^-_{1,1}&0&\dots\\
M^-_{2,0}&M^-_{2,1}&M^-_{2,2}&\dots\\
\vdots&\vdots&\vdots&\ddots
\end{bmatrix}\iota,
\end{equation}
where $\iota:l^2(\bbZ_{-})\to l^2(\bbZ_{+})$,
$\iota|m\rangle=|-1-m\rangle$. Also, the shifted transformation
operator is of the form
\begin{equation}\label{tropmshifteven}
    \cM^{(n)}_+=
\begin{bmatrix}M^-_{n,n}&0&0&\dots\\
M^-_{1+n,n}&M^-_{1+n,1+n}&0&\dots\\
M^-_{2+n,0+n}&M^-_{2+n,1+n}&M^-_{2+n,2+n}&\dots\\
\vdots&\vdots&\vdots&\ddots
\end{bmatrix}
\end{equation}
for even $n$ and
\begin{equation}\label{tropmshift}
    \cM^{(n)}_+=\begin{bmatrix}A& &\\
&A&\\
 & & \ddots
\end{bmatrix}
\begin{bmatrix}M^-_{n,n}&0&0&\dots\\
M^-_{1+n,n}&M^-_{1+n,1+n}&0&\dots\\
M^-_{2+n,0+n}&M^-_{2+n,1+n}&M^-_{2+n,2+n}&\dots\\
\vdots&\vdots&\vdots&\ddots
\end{bmatrix}\cS_n^*\fA_1^*\cS_n
\end{equation}
for odd $n$, where $\cS_n:l^2(\bbZ_{+})\to l^2(\bbZ_{+,n})$,
$\cS_n|m\rangle=|n+m\rangle$ and $
A=\begin{bmatrix}\bar a& \rho\\
 \rho&- a
\end{bmatrix}
$ is related to the matrix $\fA_a$ with constant coefficients.

\begin{lemma} $\cM_+$ is bounded if and only if
\begin{equation}\label{yasam4}
    \int_{\bbT}|F(\t)|^2dm(\t)\le C\,\Vert F\Vert^2_{\alpha_+}
\end{equation}
is satisfied for all $F\in\cF^+(\bbZ_{+})$. If  $\cM^{(n)}_+$ is
bounded for a certain $n=n_0$ then it is bounded for all $n\in\bbZ$.
\end{lemma}

\begin{proof} \eqref{yasam4} follows directly from
\eqref{transferpm}. $\|\cM^{(n+1)}_+\|\le\|\cM^{(n)}_+\|$. So the
only thing is required to be proved is that $\|\cM^{(n)}_+\|<\infty$
implies $\|\cM^{(n-1)}_+\|<\infty$. It follows from the recurrence
\eqref{evp}.

\end{proof}

Let
\begin{equation}\label{yasam1}
\begin{split}
   &\frac{1+\theta(v)}
   {1-\theta(v)}=i\Im \frac{1+\theta(0)}
   {1-\theta(0)}+\int_{\bbT}\frac{t+v}
   {t-v}d\sigma(t)\\
   =&i\Im \frac{1+\theta(0)}
   {1-\theta(0)}+\sum_{t_k\in\bbT\setminus E}\frac{t_k+v}
   {t_k-v}\sigma_k+\int_{E}\frac{t+v}
   {t-v}\left\{\fw(t)dm(t)+d\sigma_s(t)\right\},
   \end{split}
\end{equation}
where $\sigma_s$ is a singular measure on $E$ and
\begin{equation}\label{sigmaac}
\fw(t)=\frac{1-|\theta(t)|^2}{|1-\theta(t)|^2}.
\end{equation}
Then
\begin{equation}\label{yasam}
  \frac{1-\theta(v)\overline{\theta(v_0)}}
   {(1-\theta(v))(1-\overline{\theta(v_0)}}=\int_{\bbT}\frac{1-v\bar v_0}
   {(t-v)(\bar t-\bar v)}d\sigma(t).
\end{equation}

\begin{lemma} Let $\theta\in\thte$. Put $\theta_+=\theta$ and select $\theta_-$
as in  Example \ref{ex1}, so that the associate CMV matrix $\fA$
belongs to $\fae$. Then
\begin{equation}\label{yasam2}
    F(\z):=(L(\z,\k)-e^{ic}L(\z,\bar\k))\int_{\bbT}\frac{t}{b_{\bar \k}(\z)t-b_\k(\z)}
    d\sigma(t) f(t)
\end{equation}
is a unitary map from $L^2_{d\sigma}$ to $\cF^+(\bbZ_{+,1})$.
\end{lemma}

\begin{proof}
Put $f(t)=\frac{L(\k,\z_0)-e^{ic}L(\bar\k,\z_0)}{{ b_{
\k}(\bar\z_0)}-t{b_{\bar\k}(\bar\z_0)}}$. Note  that
\begin{equation}\label{yasam3}
    l^{1,+}(\z,\z_0)=\frac{L(\k,\z)L(\z_0,\k)-L
    (\bar\k,\z)L(\z_0,\bar\k)}{b_{\bar\k}(\z)\overline{b_{\bar\k}(\z_0)}
    -b_{\k}(\z)\overline{b_{\k}(\z_0)}}
\end{equation}
is the reproducing kernel in $\cF^+(\bbZ_{+,1})$. By \eqref{yasam}
we have
\begin{equation*}
    \|f\|^2_{L_\sigma^2}=\|F\|^2_{\alpha_+}.
\end{equation*}
Thus the map is an isometry. Since  the set of such functions is
dense, it is unitary.
\end{proof}

\begin{proposition}The transformation operator
$\cM_+^{(1)}$ is bounded if and only if
\begin{equation}\label{yasam5}
    \int_{E}|(\fH f)(v)|^2\frac{dm(v)}{\fw(v)}\le C\Vert
    f\Vert^2_{L^2_\sigma},\quad f\in{L^2_\sigma},
\end{equation}
where
\begin{equation}\label{yasam6}
     (\fH f)(v):=\int_{\bbT}\frac{t}{t-v}
    d\sigma(t) f(t).
\end{equation}
\end{proposition}

\begin{proof} We use \eqref{yasam4}.
Then, by \eqref{yasam2} and \eqref{wid0}, we have
\begin{equation*}
    \int_{E}\frac{|1-\theta(v)|^2}{1-|\theta(v)|^2}|(\fH f)(v)|^2dm(v)\le
C\,\|F\|^2_{\alpha_+}=
 C\,\Vert
    f\Vert^2_{L^2_\sigma}.
\end{equation*}
Thus \eqref{yasam5} is proved.
\end{proof}

We  give necessary and sufficient conditions on measure $\sigma$
that guarantee \eqref{yasam5}.

Let us reformulate our problem and change the notations slightly.  Obviously we
can straighten up by fractional linear transformation the arc $E$
and point part of $\sigma$ in such a way that $E$ becomes the
segment $[-2,2]$, points $\{\zeta_k\}$ are transformed to $\{x_k\}$
accumulating only to $-2$ and $2$, measure $\sigma$ goes to
$\tilde\sigma$, and $d\tilde\sigma= \tilde\fw\,dx$ on $[-2,2]$,
$\tilde\sigma(x_k)=\sigma_k$. It is easy to see that inequality
\eqref{yasam5} becomes equivalent to the following one

\begin{equation}\label{yasam8}
    \int_{-2}^2|(\cH f)(y)|^2\frac{dy}{\tilde\fw(y)}\le C\Vert
    f\Vert^2_{L^2_{\tilde\sigma}},\,\,\,\forall f\in L^2(d\tilde\sigma)\,,
\end{equation}
where
\begin{equation}\label{yasam9}
     (\cH f)(v):=\int_{-2}^2\frac{f(x)}{x-y}
    d\tilde\sigma(x)\,.
\end{equation}

If we choose all $f$'s from $L^2([-2,2], \tilde\fw dx)$ we get that
\eqref{yasam8} is equivalent to $\tilde\fw\in A_2 [-2,2]$. In fact,
with such test functions $f$ \eqref{yasam8} becomes

\begin{equation}\label{yasam10}
    \int_{-2}^2\bigg|\int_{-2}^2\frac{f(x)\,\tilde\fw(x)dx}{x-y}\bigg|^2\frac{dy}{\tilde\fw(y)}
\leq C\int_{-2}^2 |f(x)|^2\,\tilde\fw(x)dx,\,\,\,\forall f\in
L^2(\tilde\fw dx).
\end{equation}

Put $F:=f\tilde\fw$. Then the previous estimate becomes the
boundedness of
$$
\cH: L^2([-2,2],\tilde\fw^{-1}dx)\rightarrow
L^2([-2,2],\tilde\fw^{-1}dx).
$$
This is of course equivalent to $\tilde\fw^{-1} \in A_2[-2,2]$,
namely, to

\begin{equation}
\label{A2_22} \sup_{I, I\subset [-2,2]} \langle \tilde\fw\rangle_I
\langle \tilde\fw^{-1}\rangle_I\,<\infty,
\end{equation}
where $\langle \tilde\fw\rangle_I := \frac{1}{|I|}\int_I \tilde\fw
dx$. This is obviously the same as  $\tilde\fw \in A_2[-2,2]$.

Notice that it is easy to proof that $\tilde\fw \in A_2[-2,2]$ if
and only if $\tilde\fw$ is a restriction onto $[-2,2]$ of an $A_2$
wight on the whole real line.

\begin{lemma} Condition \eqref{yasam5} implies that
the measure $\sigma$ is absolutely continuous on the arc $E$ and
moreover $\fw\in A_2$.
\end{lemma}

Therefore, to prove Theorem \ref{th23} we have to answer the
following question: what is the property of the singular part on
$\bbT\setminus E$?

To continue with \eqref{yasam8} we write it down now for all $f\in
L^2(X,d\tilde\sigma)$:
\begin{equation}\label{yasam11}
    \int_{-2}^2\bigg|\int_{X}\frac{f(x)\,d\tilde\sigma(x)}
{x-y}\bigg|^2\frac{dy}{\tilde\fw(y)}\leq C\int_{X} |f(x)|^2\,
d\tilde\sigma(x),\,\,\,\forall f\in L^2(X,d\tilde\sigma).
\end{equation}
Let us write down the dual inequality. Fix $g\in L^2([-2,2],
\tilde\fw dx)$, we have that \eqref{yasam11} is equivalent to
$$
\sup_{\|g\|_{L^2(\tilde\fw)}\leq 1}\int_{-2}^2 g(y) \int_X\frac{f(x)
d\tilde\sigma(x)}{x-y}dy \leq
C\bigg(\int_X|f|^2\,d\tilde\sigma\bigg)^{1/2}.
$$
Thus, we can conclude that \eqref{yasam11} is equivalent to the
following inequality:

\begin{equation}
\label{dual1} \int_X \bigg|\int_{-2}^2\frac{g(y)}{y-x}dy\bigg|^2
\,d\tilde\sigma \leq C\,\int_{-2}^2 |g|^2 \tilde\fw
dx\,,\,\,\,\,\forall g\in L^2([-2,2],\tilde\fw dx)\,.
\end{equation}

\vspace{.3in}

To understand necessary and sufficient conditions for \eqref{dual1}
we introduce Smir\-nov class $E^2(\Omega)$, where
$\Omega=\bbC\setminus [-2,2]$. Recall that this is the class of
analytic functions $f$ on $\Omega$ having the property that
$\int_{\gamma_n} |f(z)|^2 |dz|\leq C$ for a sequence of smooth
contours converging to $[-2,2]$ (class does not depend on the
sequence of contours). Let us denote by $\phi(z)$ the outer function
in $\Omega$ such that $\tilde\fw=|\phi|^2$ on the boundary $[-2,2]$
of $\Omega$ (the same boundary value on both sides of $[-2,2]$),
$\phi(\infty)>0$. The fact that $\tilde\fw\in A_2[-2,2]$ is way
sufficient for the fact that $\phi$ exists (as $\tilde\fw\in
A_2[-2,2]$ obviously ensures $\int_{-2}^2 |\frac{\log
\tilde\fw(x)}{\sqrt{4-x^2}}|\,dx<\infty$, and the latter condition
means the existence of outer function in $\Omega$ with absolute
value $\tilde\fw$ on the boundary).

\vspace{.3in}

\begin{lemma}
\label{E2A2} Let $\tilde\fw\in A_2[-2,2]$,  $\int_{-2}^2 |g|^2
\tilde\fw dx<\infty$, and let
$$
G(z)=\int_{-2}^2\frac{g(t)dt}{t-z}\,.
$$
Then $G(z)\phi(z)\in E^2(\Omega)$.
\end{lemma}

\begin{proof}
Consider
$$
G_+(x):=\lim_{y\rightarrow 0+}\int_{-2}^2\frac{g(t)dt}{t-x-iy},
$$
$$
G_-(x):=\lim_{y\rightarrow 0-}\int_{-2}^2\frac{g(t)dt}{t-x-iy}
$$
Jump formula says that $G_+(x) -G_-(x) = c\cdot g(x)$ for a. e. $x$.
On the other hand, $G_+(x) + G_-(x)= c\cdot \cH g(x)$ for a. e. $x$.
We conclude that both  $G_+, G_-\in L^2([-2,2],\tilde\fw dx)$ if and
only if both $g,\cH g\in L^2([-2,2],\tilde\fw dx)$. The latter is
the same as $\cH g\in L^2([-2,2],\tilde\fw dx)$ (because of course $
g\in L^2([-2,2],\tilde\fw dx)$ by assumption). We conclude that both
boundary values are in $L^2([-2,2],\tilde\fw dx)$ if and only if
$\cH g$ is. But the latter condition is equivalent to (we discussed
this already) $\tilde\fw\in A_2[-2,2]$.

We are left to prove that $G_+, G_-\in L^2([-2,2],\tilde\fw dx)$
implies that $G(z)\phi(z)\in E^2(\Omega)$. Actually these claims are
equivalent, and this does not depend on $A_2$ anymore. Notice that
our function $G(z)$ is a Cauchy integral of $L^1(-2,2)$ function,
and, as such, belongs to Smirnov class $E^p(\Omega)$ for any $p\in
(0,1)$.

For any outer function $h$ in $\Omega$ and for any analytic function
$G$, say, from $E^{1/2}(\Omega)$ we have that $Gh\in E^2(\Omega)$ if
and only if $(Gh)_+\in L^2(-2,2), (Gh)_-\in L^2(-2,2)$. This is the
corollary of the famous theorem of Smirnov (see \cite{Privalov})
that says that if in a domain $\Omega$ one has a holomorphic
function $F$ which is the ratio of two bounded holomorphic functions
such that the denominator does not have singular inner part (the
class of such functions is denoted by $\mathcal{N}$, and if
$f|\partial{\Omega} \in L^q(\partial{\Omega})$ then $f\in
E^q(\Omega)$. In our case one should only see that any $G\in
E^{1/2}(\Omega)$  and any outer function $h$ are  functions from
$\mathcal{N}$. Then we apply this observation to our $G$ and to
outer function $h=\phi$, and we see that the requirement
$G(z)\phi(z)\in E^2(\Omega)$ is equivalent to $G_+, G_-\in
L^2([-2,2],\tilde\fw dx)$. We finished the lemma's proof.
\end{proof}

\vspace{.2in}

\begin{remark}
A little bit more is proved. Namely, given a weight $\tilde\fw$ on
$[-2,2]$, we can claim that for every $g$ such that $\int_{-2}^2
|g|^2 \tilde\fw dx<\infty$ we have that function
$$
G(z)=\int_{-2}^2\frac{g(t)dt}{t-z}
$$
satisfies $G(z)\phi(z)\in E^2(\Omega)$ if and only if $\tilde\fw\in
A_2[-2,2]$. We need this claim only in ``if" direction.
\end{remark}

Lemma \ref{E2A2} is very helpful as it allows us to write yet
another inequality equivalent to eqref{dual1}:

\begin{equation}
\label{Carleson1} \sum_{x_k\in X} |G\phi(x_k)|^2
\frac{\sigma_k}{|\phi(x_k)|^2} \leq C\, \int_{-2}^2
|G\phi(x)|^2\,dx\,.
\end{equation}

\vspace{.2in}

We want to see now that when $g$ runs over the whole of
$L^2(\tilde\fw)$, function   $G\phi$ runs over the whole of
$E^2(\Omega)$ (recall that $G(z):=
\int_{-2}^2\frac{g(t)\,dt}{t-z}$). Lemma \ref{E2A2} gives one
direction: if $g\in L^2(\tilde\fw)$ then  $G\phi\in E^2(\Omega)$.

Let us show the other inclusion. So suppose $F\in E^2(\Omega)$.
Consder $G(z)= \frac{F(z)}{\phi(z)}$. We want to represent it as
follows:
\begin{equation}
\label{repr}
\frac{F(z)}{\phi(z)}=\int_{-2}^2\frac{f(t)\,dt}{t-z}\,,\,\,f\in
L^2(\tilde\fw)\,.
\end{equation}

\vspace{.2in}

To do that notice that both boundary value functions
$\left(\frac{F(z)}{\phi(z)}\right)_+,
\left(\frac{F(z)}{\phi(z)}\right)_-$ are in $L^2(\tilde\fw)$. Here
we use again the fact that $\tilde\fw\in A_2[-2,2]$. So these two
boundary value functions are in $L^1$. And $\frac{F(z)}{\phi(z)}\in
\mathcal{N}$ of course. We use again Smirnov's theorem (see
\cite{Privalov}) to conclude that $\frac{F(z)}{\phi(z)}\in
E^1(\Omega)$. Then put
$$
f(t):= c\cdot \left(\left(\frac{F(z)}{\phi(z)}\right)_+  -
\left(\frac{F(z)}{\phi(z)}\right)_-\right)\,.
$$

\vspace{.2in}

It is in $L^2(\tilde\fw)$ and so is in $L^1$. We apply Cauchy
integral theorem to function $\frac{F(z)}{\phi(z)}$ from
$E^1(\Omega)$. We get exactly \eqref{repr} if constant $c$ is chosen
correctly.

All this reasoning shows that in \eqref{Carleson1} when $g$ runs
over the whole of $L^2(\tilde\fw)$, function   $G\phi$ runs over the
whole of $E^2(\Omega)$. Therefore \eqref{Carleson1} can be rewritten
as follows:

\begin{equation}
\label{Carleson2} \sum_{x_k\in X} |F(x_k)|^2
\frac{\sigma_k}{|\phi(x_k)|^2} \leq C\, \int_{-2}^2
|F(x)|^2\,dx\,,\,\,\,\forall F \in E^2(\Omega)\,.
\end{equation}

\vspace{.2in}

This is very nice because \eqref{Carleson2} is a familiar Carleson
measure condition, only not in Hardy class $H^2$ in the unit disc,
but for its full analog $E^2$ in $\Omega=\bbC\setminus [-2,2]$. The
trasfer from the disc to $\Omega$ is obvious:

\begin{lemma}
\label{transfer} Let $D_I$ denote two discs centered at $-2$ and $2$
and of radius $I$. Measure $d\mu$ in $\Omega$ satisfies
$$
\int |F(z)|^2 d\mu(z) \leq C\,\int_{-2}^2 |F(x)|^2\,dx
$$
for all $F\in E^2(\Omega)$ if and only if
\begin{equation}
\label{Carleson3} \int_{D_I}\frac{d\mu(z)}{\sqrt{|z^2-4|}} \leq
C'\sqrt{I}\,.
\end{equation}
\end{lemma}

\begin{proof}
Let $\psi$ be conformal map from the disc $\bbD$ onto $\Omega$. If
$F\in E^2$ then $F\circ\psi\cdot (\psi')^{1/2}\in H^2$. We apply
Carleson measure theorem to the new measure
$\tilde\mu:=\psi^{-1}*\mu$ in the disc and see that
$\tilde\mu/|\psi'|$ is a usual Carleson measure (see
\cite{Garnett}). Coming back to $\Omega$ gives \eqref{Carleson3}.
\end{proof}

Immediately we obtain the following necessary and sufficient
condition for \eqref{yasam11} (or \eqref{dual1}) to hold:

\begin{equation}
\label{NS1} \sum_{k: |x_k\pm 2| \leq \tau}
\frac{\sigma_k}{|\phi(x_k)|^2\,\sqrt{x_k^2-4}}\leq
C\sqrt{\tau}\,,\,\,\,\forall \tau > 0\,.
\end{equation}

\vspace{.4in}

The condition \eqref{NS1} plus $\tilde\fw\in A_2$ give the full
necessary and sufficient condition for \eqref{yasam8} to hold, and
so for the $L^2$ boundedness of the operators of transformation.

However we want to simplify \eqref{NS1}. The problem with this
condition as it is shown now lies in the fact that we have to
compute the outer function $\phi$ with given absolute value
$\sqrt{\tilde\fw}$ on $[-2,2]$. This might not be easy in general.
We want to use the fact that $\tilde\fw\in A_2[-2,2]$ once again to
replace $\phi(x_k)$ by a simpler expression. We need one more lemma.

\begin{lemma}
\label{comparison} Let $\tilde\fw\in A_2[-2,2]$ and let $x>2$. There
are two constants $0<c<C<\infty$ independent of $x$ such that
\begin{equation}
\label{comp} c\,\frac 1{x-2}\int_{4-x}^2 \tilde\fw dt \leq
|\phi^2(x)| \leq C\,\frac 1{x-2}\int_{4-x}^2 \tilde\fw dt\,.
\end{equation}
\end{lemma}

\begin{proof}
Let $P_z(s)$ stands for the Poisson kernel for domain $\Omega$ with
pole at $z\in \Omega$. It is easy to write its formula using
conformal mapping onto the disc, but we prefer to write its
asymptotic bahavior when $z>2$ and $z-2$ is small:
\begin{equation}
\label{asy} P_{z}(s) \asymp \frac{\sqrt{z-2}}{\sqrt{2-s}\,(z-s)}\,.
\end{equation}

\vspace{.2in}

Notice that it is sufficient to prove only the right inequality in
\eqref{comp}. In fact, the left one then follows from the right one
applied to $\tilde\fw^{-1}$ if one uses $\tilde\fw\in A_2$. So let
us have $\delta$ be a number very close to 1, but $\delta< 1$. There
exists such a $\delta$ that $\tilde\fw^{\delta}$ is still in $A_2$.

Having this in mind we write
$$
\phi^2(x)= e^{\int_{-2}^2\log \tilde\fw\, P_{x}(s)\,ds} \leq
\bigg(\int_{-2}^2 \tilde\fw^{\delta}\,\frac
1{\sqrt{2-s}}\frac{\sqrt{x-2}}{x -s}\,ds\bigg)^{\frac 1{\delta}}\,.
$$

\vspace{.2in}

We can split the last integral into two:
$$
I:= \int_{4-x}^2 \tilde\fw^{\delta}\,\frac
1{\sqrt{2-s}}\frac{\sqrt{x-2}}{x-s}\,ds\leq C\,\frac
1{\sqrt{x-2}}\int_{4-x}^2\tilde\fw^{\delta}\frac 1{\sqrt{2-s}}\,ds
$$
and
$$
II:=\int_{-2}^{4-x} \tilde\fw^{\delta}\cdot ...ds\leq C\,\int_{-2}^2
\tilde\fw^{\delta}\frac{\sqrt{x-2}}{(x-s)^{\frac 3 2}}ds\,.
$$

\vspace{.2in}

It is easy to take care of $II$. In fact, it is well known that for
any $A_2[-2,2]$ weight $\fu$
\begin{equation*}
    \int_{-2}^2 \fu(s)\frac{(x-2)^{a}}{(x-s)^{1+a}}ds \leq C_a\,\frac
1{x-2}\int_{4-x}^2 \fu(s)\,ds\,.
\end{equation*}

\vspace{.2in}

But this is false to claim that for any $A_2$ weight $\fu$ one has
$\frac 1{\sqrt{x-2}}\int_{4-x}^2 \fu\frac 1{\sqrt{2-s}}\,ds\leq
C\,\frac 1{x-2}\int_{4-x}^2 \fu(s)\,ds$!  Just take $\fu$ to be
equal to $\frac 1{\sqrt{2-s}}$ for all $s<2$ and close to $2$.
Therefore term $I$ is more difficult than term $II$. But not much.
Use Cauchy inequality:
$$
I^{\frac 1{\delta}} \leq C\,\bigg(\frac 1{\sqrt{x-2}}\int_{4-x}^2
\tilde\fw^{\delta}\frac 1{\sqrt{2-s}}\,ds\bigg)^{\frac
1{\delta}}\leq \bigg(\frac 1{x-2}\int_{4-x}^2
\tilde\fw\,ds\bigg)\cdot
$$
$$ \bigg(\frac 1{x-2}\int_{4-x}^2\frac 1{(2-s)^{\frac 1 2\cdot\frac 1{1-\delta}}}
\bigg)^{\frac{1-\delta}{\delta}}\cdot (x-2)^{\frac 1{2\delta}} \leq
C\,\frac 1{x-2}\int_{4-x}^2 \tilde\fw\,ds\,.
$$

As a result we get $|\phi^2(x)| \leq C\, \frac 1{x-2}\int_{4-x}^2
\tilde\fw\,ds$, which is the right inequality of the lemma. We
already noticed that the left inequality follows from the right one
(using $A_2$ property and applying what we proved to $\frac
1{\tilde\fw}$). Lemma is completely proved.
\end{proof}

Now we can rewrite \eqref{NS1} in an equivalent form.
\begin{proposition}
Let $x_k\rightarrow 2$ (we consider accumulation to point $2$ only,
accumulation to $-2$ is symmetric). Consider condition
\begin{equation}
\label{NS2} \sum_{k: x_k-2\leq \tau } \frac{\sigma_k}{\int_{4-x_k}^2
\tilde\fw(s)\,ds}\sqrt{x_k-2}\leq C\sqrt{\tau}\,,\,\,\,\forall \tau
> 0\,.
\end{equation}
Then (if points accumulate only to $2$) \eqref{NS2} plus
$\tilde\fw\in A_2[-2,2]$ are equivalent to \eqref{yasam8}.
\end{proposition}

If points accumulate to both $\pm 2$ we need to add  an obvious
symmetric condition near $-2$. Thus Theorem \ref{th23} is completely
proved.

\begin{remark}{\it (step backward---step forward).} Let
\begin{equation}\label{deftite1}
    \tilde\theta(v):=\frac{b_0+v\theta^{(1)}(v)}{1+\bar
b_0v\theta^{(1)}(v)},\quad b_0\in\bbD.
\end{equation}
Then
\begin{equation}\label{deftite}
    \tilde\theta(v):=\frac{1+\bar c_0\theta(0)}{1+
c_0 \overline{\theta(0)}}\frac{c_0+\theta(v)}{1+\bar c_0 \theta(v)},
\end{equation}
where $c_0=\frac{b_0-a_0}{1-\bar a_0 b_0}$ is actually an arbitrary
point in $\bbD$.  Obviously multiplication of $\tilde\theta$ by
$e^{ic}$ does not change the norm of the transformation operator.
Thus arbitrary fraction--linear transformation
\begin{equation}\label{deftite2}
    \tilde\theta(v):=e^{ic}\frac{c_0+\theta(v)}{1+\bar c_0 \theta(v)},
\end{equation}
  preserves $A_2$ \eqref{A2_220} and "Carleson" \eqref{NS20} conditions.
\end{remark}

\section{Sufficient condition in terms of scattering data}
\label{secsc}
\begin{proof}[Proof of Theorem \ref{thsc}]
Let
$$
\bW=\begin{bmatrix} 1&\bar R_+\\
R_+&1
\end{bmatrix},\quad \bB=\begin{bmatrix} \bar B& 0\\
0& B
\end{bmatrix}.
$$
Condition \eqref{sscc3} means that the matrix weight $\bB\bW\bB^*$
is in $A_2$.

First we prove that
\begin{equation}\label{2sos}
||f^-||^2\le Q||f^-||^2_{L^2_{\alpha_-}}
\end{equation}
for $f^-(t)\in \hat H^2_{\alpha_-}$. In fact even $ ||f^-||^2\le
Q||f^-||^2_{R_-}$.

Recall
$$
 T_-(\t)f^-(\t)=\bar\t f^+(\t)+R_+(\t) f^+(\t)\in \bar BH^2,
$$
where $f^+\in L^2_{\alpha_+}\ominus \check H^2_{\alpha_+}$.
Therefore
$$
P_+\bB\bW
\begin{bmatrix} f^+(\t)\\
\bar \t f^+(\bar \t)
\end{bmatrix}=\begin{bmatrix} 0\\
B(\t)T_-(\t)f^-(\t)
\end{bmatrix},
$$
and
\begin{equation}\label{1sos}
\left\langle \bB\bW^{-1}\bB^* P_+\bB\bW
\begin{bmatrix} f^+(\t)\\
\bar \t f^+(\bar \t)
\end{bmatrix},
P_+\bB\bW
\begin{bmatrix} f^+(\t)\\
\bar \t f(^+\bar \t)
\end{bmatrix}\right\rangle=||f^-||^2.
\end{equation}
Due to the $A_2$ condition we get
\end{proof}

\begin{equation}\label{1sos}
||f^-||^2 \le Q \left\langle \bW
\begin{bmatrix} f^+(\t)\\
\bar \t f^+(\bar \t)
\end{bmatrix},
\begin{bmatrix} f^+(\t)\\
\bar \t f^+(\bar \t)
\end{bmatrix}\right\rangle=Q ||f^-||^2_{R_-}.
\end{equation}

Now  we will prove the second part of the claim,  that is,
\begin{equation}\label{2sosplus}
||f^+||^2\le Q||f^+||^2_{L^2_{\alpha_+}}
\end{equation}
for $f^+(t)\in \hat H^2_{\alpha_+}$.

Since \eqref{2sos} holds then $\hat H^2_{\alpha_-}= \check
H^2_{\alpha_-}$ (moreover $\hat H^2_{\alpha_-}\subset H^2$).
Evidently, this implies $\hat H^2_{\alpha_+}= \check
H^2_{\alpha_+}$. Indeed,
\begin{equation*}
\begin{split}
\hat H^2_{\alpha_+}=&(L^2_{\alpha_-}\ominus \check H^2_{\alpha_-})^+
=(L^2_{\alpha_-}\ominus \hat H^2_{\alpha_-})^+
=L^2_{\alpha_+}\ominus (\hat H^2_{\alpha_-})^+= \check
H^2_{\alpha_+}.
\end{split}
\end{equation*}
Therefore \eqref{2sosplus} is guarantied by the inequality
\begin{equation}\label{3sos}
\langle f,f\rangle\le Q \left\{\left\langle \bW
\begin{bmatrix} f(\t)\\
\bar \t f(\bar t)
\end{bmatrix},
\begin{bmatrix} f(\t)\\
\bar \t f(\bar \t)
\end{bmatrix}\right\rangle+ \sum |f(\z_k)|^2\nu(\z_k)\right\}
\end{equation}
 for functions of the form $f=f_1+Bf_2$, where
$$
f_1(\z)=\sum_{k=1}^N\frac{B(\z)}{(\z-\z_k)B'(\z_k)}{f(\z_k)}, \quad
f_2\in H^2.
$$

Note that $f_1$ and $Bf_2$ orthogonal with respect to the standard
metric in $H^2$, i.e.,
$$
\|f\|^2=\|f_1\|^2+\|f_2\|^2.
$$
Let us calculate the matrix of the matric in $\check H^2_{\alpha_+}$
which is generated by this decomposition.
\begin{equation}\label{gmat1}
\langle Bf_2,Bf_2\rangle_{L^2_{\alpha_+}}=\frac 1 2\left\langle
\bB\bW\bB^*
\begin{bmatrix} f_2(\t)\\
\bar \t f_2(\bar \t)
\end{bmatrix},
\begin{bmatrix} f_2(\t)\\
\bar \t f_2(\bar \t)
\end{bmatrix}\right\rangle=\langle(I+\bH_2) f_2,f_2\rangle,
\end{equation}
where $\bH_2$ is the Hankel operator generated by the symbol $\tilde
R_+$,
$$
\bH_2 f_2=P_+\bar\t(\tilde R_+ f_2)(\bar\t).
$$
Similarly
\begin{equation}\label{gmat2}
\langle f_1,f_1\rangle_{L^2_{\alpha_+}} =\langle(I+\bH_1)
f_1,f_1\rangle+ \delta(f_1,f_1).
\end{equation}
Here $\delta$  is the quadratic form corresponding to the scalar
product in $L^2_{\nu_+}$. Finally,
\begin{equation}\label{gmat3}
\langle f_2,f_1\rangle_{L^2_{\alpha_+}} =\langle \bT f_2,f_1\rangle,
\end{equation}
where $\bT$ is the truncated Toeplitz operator
$$
T f_2=P_+BP_-\bar\t(\tilde R_+ f_2)(\bar\t).
$$

In these terms, according to \eqref{3sos} and the above
\eqref{gmat1}--\eqref{gmat3}, we have to show that there exists
$\epsilon(=\frac 1 Q)>0$ such that
\begin{equation}\label{mmm2}
    \epsilon\begin{bmatrix} I &0\\
    0&I\end{bmatrix}\le
    \begin{bmatrix} I+\bH_1 +\delta &\bT\\
    \bT^*&I+\bH_2\end{bmatrix}.
\end{equation}
By  \eqref{sscc3} we have $\|\bH_2\|<1$. Therefore we can substitute
\eqref{mmm2} by
\begin{equation}\label{mmm1}
    \epsilon\begin{bmatrix} I &0\\
    0&I+\bH_2\end{bmatrix}\le
    \begin{bmatrix} I+\bH_1 +\delta &\bT\\
    \bT^*&I+\bH_2\end{bmatrix}.
\end{equation}
It is equivalent to
\begin{equation*}
    \begin{bmatrix} I+\bH_1 +\delta-\epsilon &\bT\\
    \bT^*&(1-\epsilon)(I+\bH_2)\end{bmatrix}\ge 0
\end{equation*}
or
\begin{equation}\label{mmm4}\begin{split}
    &\begin{bmatrix} (1-\epsilon)(I+\bH_1 +\delta-\epsilon) &\bT\\
    \bT^*&(I+\bH_2)\end{bmatrix}\\&=
    \begin{bmatrix} I+\bH_1  &\bT\\
    \bT^*&(I+\bH_2)\end{bmatrix}
    +\begin{bmatrix} (1-\epsilon)(\delta-\epsilon)
    -\epsilon(I+\bH_1)&0\\
    0& 0\end{bmatrix}\ge 0.
    \end{split}
\end{equation}

Since the first term in the RHS \eqref{mmm4} is nonnegative it is
enough to  find $\epsilon$ such that
$$
\frac{\epsilon}{1-\epsilon}(I+\bH_1)+\epsilon I\le\epsilon
\frac{3-\epsilon}{1-\epsilon}I\le\delta.
$$
Note that the last inequality is the same as
\begin{equation}\label{mmm5}
   \epsilon \frac{3-\epsilon}{1-\epsilon}\|f_1\|^2\le
   \|f_1\|_{L_{\nu_+}^2}^2.
\end{equation}

Due to the below (well known) lemma, the Carleson condition for
$\tilde\nu_+$, given by \eqref{sscc1}, implies
\begin{equation}\label{mmm6}
    \|f_1\|^2\le Q\|f_1\|_{L_{\nu_+}^2}^2.
\end{equation}
Thus \eqref{mmm5} and consequently the whole theorem is proved.

\begin{lemma} The following condition
\begin{equation}\label{dc1}
    \|f\|_{K_B}^2\le Q\|f\|_{L_\nu^2}^2, \ \forall f\in K_B:=H^2\ominus BH^2,
\end{equation}
is satisfied if and only if $\tilde \nu$, $\tilde \nu(\z_k)=\frac
1{|B'(\z_k)|^2\nu(\z_k)}$, is a Carleson measure.
\end{lemma}

\begin{proof}
 A function $f\in K_B$ we represent in the form
$$
f=\sum\frac{B(\z)}{(\z-\z_k)B'(\z_k)}{f(\z_k)}
$$
so \eqref{dc1} is equivalent to the  boundness of the operator
$$
A(\{x_k\})=\sum\frac{B(\z)}{(\z-\z_k)B'(\z_k)}\frac{x_k}{\sqrt{\nu(\z_k)}}
 :l^2\to K_B.
$$
Note that
$$
\langle f,g\rangle=\sum \frac{x_k}{\sqrt{\nu(\z_k)}}\bar y_k \quad
 \text{for}\quad
g=\sum\frac{y_k}{1-\z\bar\z_k}.
$$
That is,
$$
A^*(g)=\left\{\frac{y_k}{\sqrt{\nu(\z_k)}}\right\}
$$
and \eqref{dc1} can be rewritten into the form
\begin{equation}\label{dc2}
\sum\frac{|y_k|^2}{\nu(\z_k)}\le Q\|g\|_{K_B}^2.
\end{equation}
Note that
$$
g=\sum\frac{y_k}{1-\z\bar\z_k}\mapsto\tilde
g=\sum\frac{B(\z)y_k}{(\z-\z_k)}
$$
is a unitary mapping, and thus we get from \eqref{dc2}
\begin{equation}\label{dc3}
    \sum|\tilde g(\z_k)|^2\frac 1{|B'(\z_k)|^2\nu(\z_k)}\le Q\|\tilde g\|^2_{K_B}
\end{equation}
for all $\tilde g\in K_B$.
\end{proof}


\section{attachment}


The space $L_{\theta}$ is defined as the set of 2$D$ vector
functions with the scalar product
\begin{equation}\label{al2}
\|f\|^2=
\int_{T}  \left\langle\begin{bmatrix}1&\theta\\
\bar\theta&1\end{bmatrix}\begin{bmatrix} f_1\\ f_2
\end{bmatrix},\begin{bmatrix} f_1\\ f_2
\end{bmatrix}\right\rangle dm(t).
\end{equation}
By $K_{\theta}$ we denote its subspace
$$
K_{\theta}= L_{\theta}\ominus\begin{bmatrix} H^2_-\\H^2_+
\end{bmatrix}.
$$
\begin{lemma}
The vector
$$
\begin{bmatrix} 1\\ -\overline{\theta(\mu)}
\end{bmatrix}\frac{1}{1-t\bar\mu}
$$
is the reproducing kernel in $K_\theta$.
\end{lemma}

\begin{proposition} Let $E=\{t: |\theta|\not=1\}$.
For $f\in K_\theta$
\begin{equation}\label{inkth}
    \int_E\frac{|f_1+\theta f_2|^2}{1-|\theta|^2}dm\le C\|f\|^2
\end{equation}
(compere \eqref{yasam5}) if and only if
\begin{equation}\label{bddth}
\int_{E}  \left\langle\begin{bmatrix}1&\theta\\
\bar\theta&1\end{bmatrix}^{-1}Hf,Hf\right\rangle dm(t)\le
C\int_{T}  \left\langle\begin{bmatrix}1&\theta\\
\bar\theta&1\end{bmatrix}f,f\right\rangle dm(t),
\end{equation}
where $(H f)(z)$ is the Hilbert Transform
$$
(H f)(z)=\int  \begin{bmatrix}1&\theta\\
\bar\theta&1\end{bmatrix}\begin{bmatrix} f_1\\ f_2
\end{bmatrix}\frac{t}{t-z}dm(t).
$$

\end{proposition}
\begin{proof}
By definition
\begin{equation}\label{apr1}
    \begin{split} Hf(z)=&\int  \begin{bmatrix}1&\theta\\
\bar\theta&1\end{bmatrix}\left\{\begin{bmatrix} f_1\\ f_2
\end{bmatrix}+\begin{bmatrix} h_-\\ h_+
\end{bmatrix}\right\}\frac{t}{t-z}dm(t) \\
    =& P_+\left\{\begin{bmatrix} f_1+\theta f_2\\ \bar\theta f_1+f_2
\end{bmatrix}+\begin{bmatrix} 1\\ \bar\theta
\end{bmatrix}h_-+\begin{bmatrix} \theta\\ 1
\end{bmatrix} h_+
\right\}(z)\\
=&\begin{bmatrix} (f_1+\theta f_2)(z)\\ 0
\end{bmatrix}+\begin{bmatrix} \theta(z)\\ 1
\end{bmatrix} h_+(z).
    \end{split}
\end{equation}
Therefore
\begin{equation}\label{bddth2}
   \begin{split}
&\int_{E}  \left\langle\frac{\begin{bmatrix}1&-\theta\\
-\bar\theta&1\end{bmatrix}}{1-|\theta|^2}Hf,Hf\right\rangle dm\\=&
\int_{E}  \left\langle\frac{\begin{bmatrix}1\\
-\bar\theta\end{bmatrix}(f_1+\theta
f_2)}{1-|\theta|^2}+\begin{bmatrix} 0\\ 1
\end{bmatrix}h_+,
\begin{bmatrix} f_1+\theta f_2\\ 0
\end{bmatrix}+\begin{bmatrix} \theta\\ 1
\end{bmatrix} h_+
\right\rangle dm\\
=&\int_E\frac{|f_1+\theta f_2|^2}{1-|\theta|^2}dm+\int_E|h_+|^2 dm.
 \end{split}
\end{equation}
That is \eqref{bddth} is equivalent to
\begin{equation}\label{inkth2}
    \int_E\frac{|f_1+\theta f_2|^2}{1-|\theta|^2}dm+\int_E|h_+|^2 dm\le C(\|f\|^2
+\int_\bbT|h_+|^2 dm).
\end{equation}
\end{proof}

\begin{lemma} Condition \eqref{bddth} implies
\begin{equation}\label{sina2}
\sup_{I}\frac{1}{|I|}\int_{I} \frac{|\theta-\Mth|^2+
(1-|\Mth|^2)}{1-|\theta|^2} \,dm<\infty,
\end{equation}
where for an arc $I\subset E$ we put
\begin{equation}
\Mth:=\frac{1}{|I|}\int_{I} \theta\,dm.
\end{equation}
\end{lemma}
\begin{proof} In particular \eqref{bddth} implies that the matrix
weight $\begin{bmatrix}1&\theta\\
\bar\theta&1\end{bmatrix}$ is in $A_2$ on $E$. And this means
\begin{equation}\label{aa3}
    \int_I \frac{\begin{bmatrix}1&-\theta\\
-\bar\theta&1\end{bmatrix}}{1-|\theta|^2}dm\le C\begin{bmatrix}1&\Mth\\
\bar{\Mth}&1\end{bmatrix}^{-1},
\end{equation}
or
\begin{equation}\label{aa4}
    \int_I \frac{\begin{bmatrix}1-|\theta|^2+|\theta-\Mth|^2&(\Mth-\theta)\sqrt{1-|\Mth|^2}\\
\overline{(\Mth-\theta)}\sqrt{1-|\Mth|^2}&
1-|\Mth|^2\end{bmatrix}}{1-|\theta|^2}dm\le C,
\end{equation}
which is equivalent to \eqref{sina2}
\end{proof}

\bibliographystyle{amsplain}

\end{document}